\documentclass[12pt]{amsart}
\usepackage[T1]{fontenc}
\usepackage{amsmath,amssymb,amsthm,mathrsfs,txfonts}
\usepackage[a4paper,margin=25mm]{geometry}
\usepackage[hidelinks]{hyperref}
\usepackage{etoolbox}
\usepackage{needspace}
\makeatletter
\patchcmd{\@settitle}{\uppercasenonmath\@title}{}{}{}
\patchcmd{\@setauthors}{\MakeUppercase{\authors}}{\authors}{}{}
\apptocmd{\@setauthors}{\firstpageauthorinfo}{}{}
\@ifundefined{@setaddresses}{}{\let\@setaddresses\@empty}
\renewcommand{\subsection}{\@startsection{subsection}{2}
  {\z@}{.5\baselineskip\@plus.7\baselineskip}{.25\baselineskip}
  {\normalfont\bfseries}}
\makeatother
\allowdisplaybreaks
\numberwithin{equation}{section}
\newtheorem{theorem}{Theorem}[section]
\newtheorem{proposition}[theorem]{Proposition}
\newtheorem{corollary}[theorem]{Corollary}
\newtheorem{lem}[theorem]{Lemma}
\theoremstyle{definition}
\newtheorem{definition}[theorem]{Definition}
\theoremstyle{remark}

\newtheorem{example}[theorem]{Example}

\begin{document}
\title[Bilinear pseudodifferential operators]
{Boundedness of nonregular bilinear pseudodifferential operators
on variable Besov and Triebel--Lizorkin spaces}
\author[T. Bai]{Tengfei Bai$^{1}$}
\author[J. Xu]{Jingshi Xu$^{2,3}$}
\author[J. Zhu]{Jinlai Zhu$^{4,\ast}$}
\thanks{$\ast$ Corresponding author.}
\date{}

\newcommand{\firstpageauthorinfo}{%
\begin{center}
\begin{minipage}{0.96\textwidth}
\centering\footnotesize
\textsuperscript{1}School of Mathematics and Statistics, Hainan Normal University,
Haikou, Hainan 571158, China\par
\smallskip
\textsuperscript{2}School of Mathematics and Computing Science, Guangxi Colleges
and Universities Key Laboratory of Data Analysis and Computation,
Guilin University of Electronic Technology, Guilin, 541004, China\par
\smallskip
\textsuperscript{3}Center for Applied Mathematics of Guangxi (GUET),
Guilin, 541004, China\par
\smallskip
\textsuperscript{4}School of Computer Engineering, Guangzhou City University of
Technology, Guangzhou 510800, China\par
\medskip
\textit{E-mails}: \href{mailto:202311070100007@hainnu.edu.cn}{\texttt{202311070100007@hainnu.edu.cn}},
\href{mailto:baitengfei1@126.com}{\texttt{baitengfei1@126.com}} (T. Bai)\par
\href{mailto:jingshixu@126.com}{\texttt{jingshixu@126.com}} (J. Xu)\par
\href{mailto:zhujl@gcu.edu.cn}{\texttt{zhujl@gcu.edu.cn}} (J. Zhu)
\end{minipage}
\end{center}%
}

\begin{abstract}
In this article, we introduce two bilinear, non-regular symbol classes\linebreak
$BS\mathcal B^m_\delta(r,\mu,v;N,\lambda)$ and
$BS\mathcal F^m_\delta(r,\mu,v;N,\lambda)$, where $m\in\mathbb R$,
$0\leq\delta<1$, $0<\mu,v\leq\infty$, $1\leq\lambda\leq2$,
$r>n/\mu$, and $N>2n/\lambda$, with $\mu<\infty$ for the latter class.
For $\mu\geq1$, $(1-\delta)r\geq n/\mu$, and sufficiently large $N$,
we establish boundedness of the associated bilinear pseudodifferential
operators on Besov and Triebel--Lizorkin spaces with variable smoothness,
integrability and summation exponents.
To the best of the authors' knowledge, these operator estimates are new
even on the classical Besov and
Triebel--Lizorkin spaces with constant exponents.
As an application, we establish the multiplication algebra property
for both variable exponent spaces.
\end{abstract}
\maketitle

\section{Introduction}\label{sec:introduction}

Pseudodifferential operators with nonregular symbols arise naturally
when the coefficients have only finite smoothness. In this setting,
the regularity of the symbol in the spatial variable and its regularity
in the frequency variables have to be distinguished. Marschall
\cite{11,12} studied symbols defined by vector-valued Besov and
Triebel--Lizorkin spaces and established boundedness of the associated
linear pseudodifferential operators on Besov and Triebel--Lizorkin
spaces. These symbol classes allow spatial regularity to be expressed
by a function space norm, rather than by bounds for derivatives of
every order. They also contain the classical smooth H\"ormander
symbols for suitable choices of the parameters.

For a measurable symbol $a(x,\xi,\eta)$ on $\mathbb R^{3n}$, the
associated bilinear pseudodifferential operator is given by
\[
T_a(f,g)(x)=\int_{\mathbb R^{2n}}
e^{2\pi ix\cdot(\xi+\eta)}a(x,\xi,\eta)
\widehat f(\xi)\widehat g(\eta)\,\mathrm d\xi\,\mathrm d\eta,
\qquad f,g\in\mathcal S(\mathbb R^n),
\]
whenever the integral is absolutely convergent. The boundedness of
these operators depends on the interaction between the two input
frequencies. B\'enyi, Bernicot, Maldonado, Naibo and Torres \cite{b}
studied bilinear H\"ormander classes and obtained boundedness results
for the associated operators. Naibo \cite{NO} established estimates
on nonhomogeneous Besov and Triebel--Lizorkin spaces of positive
smoothness. In a nonregular setting, Herbert and Naibo \cite{HN}
proved Lebesgue space estimates for symbols belonging to Besov spaces
of product type. Their results show that finite regularity conditions
on a bilinear symbol can be sufficient for boundedness.

Besov and Triebel--Lizorkin spaces with variable smoothness and
integrability provide a natural setting for estimates whose local
regularity or integrability requirements depend on the spatial point.
Early work on variable smoothness includes Leopold's study of
Besov-type spaces of variable order of differentiation \cite{Leopold89}
and Besov's results on equivalent norms, interpolation, embeddings and
extensions for spaces of variable smoothness \cite{Besov03,Besov05}.
Xu \cite{Xu08,Xu09} studied Besov and Triebel--Lizorkin spaces with
variable integrability $p(\cdot)$ and constant smoothness and summation
exponents, obtaining equivalent norms and atomic decompositions.
Diening, H\"ast\"o and Roudenko \cite{DHR} introduced the
Triebel--Lizorkin scale with all three indices variable, and Almeida
and H\"ast\"o \cite{c} introduced the corresponding Besov scale;
these works also established independence of the admissible dyadic
resolution.

Kempka \cite{Kempka09,Kempka10} developed a more general
$2$-microlocal framework with variable integrability, including local
means and atomic, molecular and wavelet characterizations.
Vyb\'iral \cite{Vybiral09} proved Sobolev and Jawerth embeddings for
Triebel--Lizorkin spaces with variable smoothness and integrability.
Kempka and Vyb\'iral \cite{n} obtained characterizations by local
means, ball means of differences and Peetre maximal functions.
Almeida and Caetano \cite{AC} revisited the Peetre maximal
characterizations and dyadic resolution independence in the
$2$-microlocal setting with all exponents variable, and established
lifting properties, embeddings and Fourier multiplier results.
We use these characterizations and convolution estimates below to
compare the dyadic norms and to estimate localized functions.

Several authors have considered pseudodifferential operators on these
variable exponent spaces. Drihem and Hebbache \cite{DB,DHC} proved
boundedness of linear operators with symbols measured by vector-valued
Besov norms, respectively on variable Besov and Triebel--Lizorkin
spaces. Their symbol classes include classical smooth symbols.
For bilinear operators, Xu and Zhu \cite{XE} obtained Leibniz-type
estimates associated with bilinear H\"ormander classes. Liu and
Wang \cite{LT} studied the smooth class $BS^m_{1,1}$ on both variable
Besov and Triebel--Lizorkin spaces. These works motivate the study of
bilinear operators whose spatial regularity is finite and whose joint
frequency regularity is measured by a Besov norm.

We introduce the nonregular symbol classes
$BS\mathcal B^m_\delta(r,\mu,v;N,\lambda)$ and
$BS\mathcal F^m_\delta(r,\mu,v;N,\lambda)$.
For each joint frequency index $k\in\mathbb N_0$, both input frequencies
are normalized at the scale $2^k$. The normalized symbol pieces have
uniform $B^N_{\lambda,\infty}(\mathbb R^{2n})$ bounds of order $2^{km}$
and spatial Besov or Triebel--Lizorkin bounds of order
$2^{k(m+\delta r)}$. Thus the regularity in $x$ and in $(\xi,\eta)$
can be prescribed separately.

Theorem~\ref{thm:triebel} gives two-term Leibniz estimates on the
variable Besov and Triebel--Lizorkin spaces for both symbol classes.
The bounded variable exponents $p,p_1,p_2,q$ have globally log-H\"older
continuous reciprocals and satisfy
\[
\frac1{p(x)}=\frac1{p_1(x)}+\frac1{p_2(x)},\qquad
p^-,p_1^-,p_2^-,q^->1,\qquad x\in\mathbb R^n.
\]
For bounded, locally log-H\"older continuous smoothness $s$, the spatial
conditions are
\[
(1-\delta)r\geq\frac n\mu,\qquad
0<s^-\leq s^+<r-\frac n\mu.
\]
Under these conditions and sufficiently large joint frequency regularity
$N$, we prove \eqref{eq:triebel-main} and \eqref{eq:besov-main}:
in each term, one input carries the smoothness $s(\cdot)+m$, and the
other is measured in a variable Lebesgue space.
The precise assumptions on the symbol parameters are given in
Theorem~\ref{thm:triebel}. If $s^-+m>0$, the estimates extend by density
to continuous bilinear maps on the corresponding products of function
spaces, as stated in Corollary~\ref{cor:extension}.

We also apply the dyadic argument to pointwise multiplication, allowing
the smoothness, integrability and summation exponents all to vary.
Theorem~\ref{thm:algebra} proves that
$B^{s(\cdot)}_{p(\cdot),q(\cdot)}$ and
$F^{s(\cdot)}_{p(\cdot),q(\cdot)}$ are multiplication algebras when
$s(\cdot)-n/p(\cdot)$ has positive essential infimum.

The symbol classes also retain the classical smooth theory. For
$m\in\mathbb R$, $0\leq\delta<1$, $r>0$, $1\leq\lambda\leq2$, and
$N>2n/\lambda$, Corollary~\ref{cor:strict-symbol} shows that
\[
BS^m_{1,\delta}\subsetneq
BS\mathcal B^m_\delta(r,\infty,\infty;N,\lambda).
\]
Example~\ref{ex:rough-both} gives symbols in the Besov symbol class
with finite regularity in both the spatial and joint frequency variables
that are not $C^\infty$ in either group of variables.

The paper is organized as follows. Section~\ref{sec:definitions} contains
the definitions and main results. Section~\ref{sec:symbol-properties}
establishes independence of the Besov symbol quasi-norm from the dyadic
resolutions and joint frequency cutoffs, parameter embeddings, cutoff
stability, and the relation to classical smooth symbols.
Section~\ref{sec:main-proof} collects the preliminary estimates and
proves Theorems~\ref{thm:triebel} and~\ref{thm:algebra} and
Corollary~\ref{cor:extension}.

\section{Preliminaries and main results}\label{sec:definitions}

Throughout the paper, $n\in\mathbb N$, $\mathbb N_0=\{0,1,2,\ldots\}$,
and all function spaces are defined on $\mathbb R^n$ unless another
dimension is indicated. We write $u\lesssim v$ if $u\leq Cv$ with a
constant independent of the functions and dyadic indices under
consideration. The relation $u\asymp v$ means $u\lesssim v\lesssim u$.
We use $1/\infty=0$.
For each $d\in\mathbb N$, we write
\[
C_b^\infty(\mathbb R^d)
:=\left\{f\in C^\infty(\mathbb R^d):
\left\|\partial^\alpha f\right\|_{L^\infty(\mathbb R^d)}<\infty
\text{ for every }\alpha\in\mathbb N_0^d\right\}.
\]
\subsection{Variable exponents and dyadic decompositions}\label{subsec:variable-spaces}

Following \cite{c,DHR}, we use the following notation for bounded
variable exponents. For a measurable exponent $p:\mathbb R^n\to(0,\infty)$, put
\[
p^-=\mathop{\rm ess\,inf}_{x\in\mathbb R^n}p(x),
\qquad p^+=\mathop{\rm ess\,sup}_{x\in\mathbb R^n}p(x).
\]
We write $p\in\mathcal P_0$ when $0<p^-\leq p^+<\infty$.
The modular and Luxemburg quasi-norm are
\[
\varrho_{p(\cdot)}(f)=\int_{\mathbb R^n}\left|f(x)\right|^{p(x)}\,\mathrm d x,
\qquad
\left\|f\right\|_{L^{p(\cdot)}}=
\inf\left\{\lambda>0:\varrho_{p(\cdot)}(f/\lambda)\leq1\right\}.
\]
The space $L^{p(\cdot)}$ consists of measurable functions having finite
quasi-norm.

A bounded real-valued function $u$ is locally log-H\"older continuous
if there exists a constant $c_{\log}(u)\geq0$ such that
\[
\left|u(x)-u(y)\right|\leq\frac{c_{\log}(u)}{\log(e+1/\left|x-y\right|)},\qquad x\ne y.
\]
It is globally log-H\"older continuous if, in addition, there exist
$u_\infty\in\mathbb R$ and $c_\infty(u)\geq0$ with
\[
\left|u(x)-u_\infty\right|\leq\frac{c_\infty(u)}{\log(e+\left|x\right|)}.
\]
The corresponding classes are denoted by $C^{\log}_{\rm loc}$ and
$C^{\log}$. We write $p\in\mathcal P_0^{\log}$ if $p\in\mathcal P_0$
and $1/p\in C^{\log}$.

Let $p,q\in\mathcal P_0$ and let $(u_k)_{k\in\mathbb N_0}$ be a
sequence of measurable functions on $\mathbb R^n$. The mixed space
$L^{p(\cdot)}(\ell^{q(\cdot)})$ consists of those sequences for which
\[
x\longmapsto\left(\sum_{k=0}^\infty \left|u_k(x)\right|^{q(x)}\right)^{1/q(x)}
\]
belongs to $L^{p(\cdot)}(\mathbb R^n)$. Its quasi-norm is
\[
\left\|(u_k)_{k\in\mathbb N_0}\right\|_{L^{p(\cdot)}(\ell^{q(\cdot)})}
=\left\|\left(\sum_{k=0}^\infty\left|u_k(\cdot)\right|^{q(\cdot)}\right)^{1/q(\cdot)}
\right\|_{L^{p(\cdot)}}.
\]
The mixed space $\ell^{q(\cdot)}(L^{p(\cdot)})$ is defined by the
modular
\[
\varrho_{\ell^{q(\cdot)}(L^{p(\cdot)})}((u_k)_{k\in\mathbb N_0})
=\sum_{k=0}^\infty\inf\left\{\lambda_k>0:
\varrho_{p(\cdot)}\left(u_k/\lambda_k^{1/q(\cdot)}\right)\leq1\right\}.
\]
It consists of the sequences for which this modular is finite after
multiplication by a positive constant. Its Luxemburg quasi-norm is
\[
\left\|(u_k)_{k\in\mathbb N_0}\right\|_{\ell^{q(\cdot)}(L^{p(\cdot)})}
=\inf\left\{\mu>0:
\varrho_{\ell^{q(\cdot)}(L^{p(\cdot)})}((u_k/\mu)_{k\in\mathbb N_0})\leq1\right\}.
\]

We use the Fourier transform
\[
\widehat f(\xi)=\int_{\mathbb R^n}e^{-2\pi ix\cdot\xi}f(x)\,\mathrm d x
\]
and its usual extension to the Schwartz space $\mathcal S(\mathbb R^n)$
and the tempered distributions $\mathcal S'(\mathbb R^n)$.
Choose $\psi\in C_c^\infty(\mathbb R^n)$, radial, with $0\leq\psi\leq1$,
strictly decreasing as a function of $\left|\xi\right|$ on $1<\left|\xi\right|<2$, and with
$\psi=1$ on $\{\left|\xi\right|\leq1\}$ and $\psi=0$ on $\{\left|\xi\right|\geq2\}$.
Define the Schwartz functions $\varphi_k$ by
\[
\widehat\varphi_0(\xi)=\psi(\xi),\qquad
\widehat\varphi_k(\xi)=\psi(2^{-k}\xi)-\psi(2^{1-k}\xi),\quad k\geq1.
\]
Then $\sum_{k=0}^{\infty}\widehat\varphi_k=1$ and
\[
\operatorname{supp}\widehat\varphi_k
\subset\left\{\xi:2^{k-1}\leq\left|\xi\right|\leq2^{k+1}\right\},\quad k\geq1.
\]
The dyadic components of $f\in\mathcal S'$ are written explicitly as
$(\varphi_k*f)=\mathcal F^{-1}(\widehat\varphi_k\widehat f)$.
For $k\in\mathbb N_0$, we have
\[
\sum_{j=0}^k\varphi_j*f
=\mathcal F^{-1}\left(\psi(2^{-k}\cdot)\widehat f\right),\qquad k\geq0.
\]
\begin{definition}\label{def:spaces}
Let $p,q\in\mathcal P_0^{\log}$ and
$s\in C^{\log}_{\rm loc}\cap L^\infty$. Write
$s^-=\operatorname*{ess\,inf}_{x\in\mathbb R^n}s(x)$ and
$s^+=\operatorname*{ess\,sup}_{x\in\mathbb R^n}s(x)$. The variable exponent
Triebel--Lizorkin and Besov spaces consist of the distributions
$f\in\mathcal S'$ for which the respective quasi-norm is finite:
\[
\begin{split}
\left\|f\right\|_{F^{s(\cdot)}_{p(\cdot),q(\cdot)}}
&=\left\|(2^{ks(\cdot)}(\varphi_k*f))_{k\in\mathbb N_0}\right\|_{L^{p(\cdot)}(\ell^{q(\cdot)})},\\
\left\|f\right\|_{B^{s(\cdot)}_{p(\cdot),q(\cdot)}}
&=\left\|(2^{ks(\cdot)}(\varphi_k*f))_{k\in\mathbb N_0}\right\|_{\ell^{q(\cdot)}(L^{p(\cdot)})}.
\end{split}
\]
\end{definition}

The spaces in Definition~\ref{def:spaces} are independent of the
choice of the admissible dyadic resolution $(\varphi_k)_{k\in\mathbb N_0}$, in the
sense that different admissible choices give equivalent quasi-norms.
See \cite[Corollary~3.2]{AC}, applied to the weight sequence
$w_k(x)=2^{ks(x)}$; the finite upper bounds on $p$ and $q$ required
for the Triebel--Lizorkin case hold by the definition of $\mathcal P_0$.

Choose constants $c_s,c_q>0$ such that, for $x\ne y$,
\[
\left|s(x)-s(y)\right|\leq\frac{c_s}{\log(e+1/\left|x-y\right|)},\qquad
\left|\frac{1}{q(x)}-\frac{1}{q(y)}\right|
\leq\frac{c_q}{\log(e+1/\left|x-y\right|)}.
\]
After increasing $c_s$ if necessary, we also have
\[
2^{ks(x)}\leq C2^{ks(y)}(1+2^k\left|x-y\right|)^{c_s},\qquad k\geq0.
\]
Here $C>0$ is independent of $x$, $y$, and $k$.

\subsection{Nonregular bilinear symbols}

For any $d\in\mathbb N$, the notation
$(\varphi^{(d)}_j)_{j\in\mathbb N_0}$ denotes a dyadic resolution
constructed in $\mathbb R^d$ in the same way as above; when $d=n$,
we use $\varphi^{(n)}_j=\varphi_j$.
We denote convolution in the spatial variable $x$ and in the joint
frequency variables $(\xi,\eta)$ by $*_x$ and $*_{(\xi,\eta)}$,
respectively. For a measurable symbol $b:\mathbb R^{3n}\to\mathbb C$,
these convolutions are given by the following formulas whenever the
integrals converge absolutely:
\[
\begin{gathered}
(\varphi_\ell*_x b)(x,\xi,\eta)
=\int_{\mathbb R^n}\varphi_\ell(x-y)b(y,\xi,\eta)\,\mathrm d y,\\
\left(\varphi_j^{(2n)}*_{(\xi,\eta)}b\right)(x,\xi,\eta)
=\int_{\mathbb R^n}\int_{\mathbb R^n}
\varphi_j^{(2n)}(\xi-u,\eta-v)b(x,u,v)\,\mathrm d u\,\mathrm d v,
\end{gathered}
\]
For tempered distributions, we use distributional convolution.
For $N>0$ and $1\leq\lambda\leq2$, the ordinary nonhomogeneous
Besov norm in the joint frequency variables $(\xi,\eta)\in\mathbb R^{2n}$ is
\[
\left\|u\right\|_{B^N_{\lambda,\infty}(\mathbb R^{2n})}
=\sup_{j\in\mathbb N_0}2^{jN}
\left\|\varphi^{(2n)}_j*_{(\xi,\eta)}u\right\|_{L^\lambda(\mathbb R^{2n})},
\qquad u\in\mathcal S'(\mathbb R^{2n}).
\]

For a Banach space $E$ and $d\in\mathbb N$, we denote by
$\mathcal S'(\mathbb R^d;E)$ the continuous linear maps from
$\mathcal S(\mathbb R^d)$ to $E$. If $b\in\mathcal S'(\mathbb R^d;E)$,
its spatial dyadic components are defined by
\[
(\varphi^{(d)}_\ell*b)(x)
=b\!\left(\varphi^{(d)}_\ell(x-\cdot)\right),
\qquad x\in\mathbb R^d,\quad \ell\in\mathbb N_0.
\]
For $d=n$, we denote this convolution by $*_x$, as above.
Each component is a smooth $E$-valued function, and its $E$-norm is
therefore measurable.

\begin{definition}\label{def:mixed-symbol}
Let $r>0$, $0<\mu,v\leq\infty$, $1\leq\lambda\leq2$, and $N>0$.
Put $E=B^N_{\lambda,\infty}(\mathbb R^{2n})$, and let
$\varphi_\ell=\varphi_\ell^{(n)}$ be the dyadic kernels on
$\mathbb R^n$ fixed above. For $b\in\mathcal S'(\mathbb R^n;E)$, define
\[
\begin{split}
\left\|b\right\|_{B^r_{\mu,v}(B^N_{\lambda,\infty})}
&=\left\|\left(2^{\ell r}
\left\|(\varphi_\ell*_x b)(x,\cdot,\cdot)\right\|_{B^N_{\lambda,\infty}(\mathbb R^{2n})}
\right)_{\ell\in\mathbb N_0}\right\|_{\ell^v(L^\mu)},\\
\left\|b\right\|_{F^r_{\mu,v}(B^N_{\lambda,\infty})}
&=\left\|\left(2^{\ell r}
\left\|(\varphi_\ell*_x b)(x,\cdot,\cdot)\right\|_{B^N_{\lambda,\infty}(\mathbb R^{2n})}
\right)_{\ell\in\mathbb N_0}\right\|_{L^\mu(\ell^v)}.
\end{split}
\]
The corresponding mixed spaces consist of the distributions for which
the displayed quasi-norm is finite, with the usual modifications when
$\mu=\infty$ or $v=\infty$. In the Triebel--Lizorkin case we require
$\mu<\infty$. Here $*_x$ is the distributional convolution defined above.
\end{definition}

Choose a radial function $\Psi\in C_c^\infty(\mathbb R^{2n})$ satisfying
\[
0\leq\Psi\leq1,\qquad
\Psi(\xi,\eta)=
\begin{cases}
1,& \left|(\xi,\eta)\right|\leq1,\\
0,& \left|(\xi,\eta)\right|\geq2,
\end{cases}
\qquad \left|(\xi,\eta)\right|=(\left|\xi\right|^2+\left|\eta\right|^2)^{1/2},
\]
and define
\[
\chi_0(\xi,\eta)=\Psi(\xi,\eta),\qquad
\chi_k(\xi,\eta)=\Psi(2^{-k}\xi,2^{-k}\eta)
-\Psi(2^{1-k}\xi,2^{1-k}\eta),\quad k\geq1.
\]
For a measurable symbol $a:\mathbb R^{3n}\to\mathbb C$, write
\begin{equation}\label{eq:profile}
b_k(x,\xi,\eta)=a(x,2^k\xi,2^k\eta)
\chi_k(2^k\xi,2^k\eta),\qquad k\geq0.
\end{equation}
For every $x\in\mathbb R^n$,
\[
\begin{split}
\operatorname{supp}b_0(x,\cdot,\cdot)
&\subset\left\{(\xi,\eta)\in\mathbb R^{2n}:\left|(\xi,\eta)\right|\leq2\right\},\\
\operatorname{supp}b_k(x,\cdot,\cdot)
&\subset\left\{(\xi,\eta)\in\mathbb R^{2n}:1/2\leq\left|(\xi,\eta)\right|\leq2\right\},
\quad k\geq1.
\end{split}
\]

\begin{definition}\label{def:symbol}
Let $m\in\mathbb R$, $0\leq\delta<1$, $0<\mu,v\leq\infty$,
$1\leq\lambda\leq2$, $r>n/\mu$, and $N>2n/\lambda$.
Let $a:\mathbb R^{3n}\to\mathbb C$ be measurable, and let
\[
b_k(x,\xi,\eta)=a(x,2^k\xi,2^k\eta)\chi_k(2^k\xi,2^k\eta),
\qquad k\in\mathbb N_0,
\]
as in \eqref{eq:profile}. The class
$BS\mathcal B^m_\delta(r,\mu,v;N,\lambda)$ consists of such symbols
$a$ satisfying
\begin{gather}
\sup_{k\geq0}2^{-km}
\mathop{\rm ess\,sup}_{x\in\mathbb R^n}
\left\|b_k(x,\cdot,\cdot)\right\|_{B^N_{\lambda,\infty}(\mathbb R^{2n})}<\infty,
\label{eq:symbol-uniform}\\*
\sup_{k\geq0}2^{-k(m+\delta r)}
\left\|b_k\right\|_{B^r_{\mu,v}(B^N_{\lambda,\infty})}<\infty.
\label{eq:symbol-spatial}
\end{gather}
For $0<\mu<\infty$, the class
$BS\mathcal F^m_\delta(r,\mu,v;N,\lambda)$ is defined by
\eqref{eq:symbol-uniform} and the condition obtained from
\eqref{eq:symbol-spatial} by replacing $B^r_{\mu,v}$ with $F^r_{\mu,v}$.
The respective quasi-norms are
\[
\left\|a\right\|_{BS\mathcal B^m_\delta(r,\mu,v;N,\lambda)}
=\sup_{k\geq0}2^{-km}\mathop{\rm ess\,sup}_{x}
\left\|b_k(x,\cdot,\cdot)\right\|_{B^N_{\lambda,\infty}}
+\sup_{k\geq0}2^{-k(m+\delta r)}
\left\|b_k\right\|_{B^r_{\mu,v}(B^N_{\lambda,\infty})},
\]
\[
\left\|a\right\|_{BS\mathcal F^m_\delta(r,\mu,v;N,\lambda)}
=\sup_{k\geq0}2^{-km}\mathop{\rm ess\,sup}_{x}
\left\|b_k(x,\cdot,\cdot)\right\|_{B^N_{\lambda,\infty}}
+\sup_{k\geq0}2^{-k(m+\delta r)}
\left\|b_k\right\|_{F^r_{\mu,v}(B^N_{\lambda,\infty})}.
\]
\end{definition}

We verify that the mixed quasi-norms in Definition~\ref{def:symbol}
are well-defined for the functions $b_k$ associated with a measurable
symbol $a$. Assume \eqref{eq:symbol-uniform}, and put
$E=B^N_{\lambda,\infty}(\mathbb R^{2n})$.
Since $N>2n/\lambda$, Bernstein's inequality gives, for $u\in E$,
\[
\begin{aligned}
&\left\|u\right\|_{L^\infty(\mathbb R^{2n})}
\leq\sum_{j=0}^{\infty}
\left\|\varphi_j^{(2n)}*_{(\xi,\eta)}u\right\|_{L^\infty(\mathbb R^{2n})}
\lesssim\sum_{j=0}^{\infty}2^{2nj/\lambda}
\left\|\varphi_j^{(2n)}*_{(\xi,\eta)}u\right\|_{L^\lambda(\mathbb R^{2n})}\\
&\quad \leq\left\|u\right\|_E\sum_{j=0}^{\infty}2^{-j(N-2n/\lambda)}
\lesssim\left\|u\right\|_E.
\end{aligned}
\]
The dyadic series converges uniformly, so $u$ has a bounded continuous
representative. We use these representatives for
$b_k(x,\cdot,\cdot)$ for almost every $x$.
For $\omega\in\mathcal S(\mathbb R^n)$, define $b_k(\omega)$ by
the scalar integral
\[
b_k(\omega)(\xi,\eta)
=\int_{\mathbb R^n}b_k(x,\xi,\eta)\omega(x)\,\mathrm d x.
\]
We show that this function belongs to $E$. Minkowski's inequality gives
\[
\begin{aligned}
&\left\|b_k(\omega)\right\|_E
=\sup_{j\in\mathbb N_0}2^{jN}
\left\|\varphi_j^{(2n)}*_{(\xi,\eta)}b_k(\omega)\right\|_{L^\lambda}
\leq\sup_{j\in\mathbb N_0}
\int_{\mathbb R^n}\left|\omega(x)\right|2^{jN}
\left\|\varphi_j^{(2n)}*_{(\xi,\eta)}b_k(x,\cdot,\cdot)\right\|_{L^\lambda}
\,\mathrm d x\\
&\quad \leq\int_{\mathbb R^n}\left|\omega(x)\right|
\left\|b_k(x,\cdot,\cdot)\right\|_E\,\mathrm d x
\leq\left\|\omega\right\|_{L^1}
\mathop{\rm ess\,sup}_{x\in\mathbb R^n}
\left\|b_k(x,\cdot,\cdot)\right\|_E<\infty.
\end{aligned}
\]
Thus $\omega\mapsto b_k(\omega)$ is a continuous linear map from
$\mathcal S(\mathbb R^n)$ to $E$. We identify $b_k$ with this
$E$-valued tempered distribution.
Moreover, for $\ell\in\mathbb N_0$,
\[
(\varphi_\ell*_x b_k)(x,\xi,\eta)
=\left[b_k\!\left(\varphi_\ell(x-\cdot)\right)\right](\xi,\eta)
=\int_{\mathbb R^n}\varphi_\ell(x-y)b_k(y,\xi,\eta)\,\mathrm d y.
\]
Consequently, the distributional convolution agrees with the scalar
integral, and both mixed quasi-norms in Definition~\ref{def:symbol}
are well-defined.

For either $\mathcal A=\mathcal B$ or $\mathcal A=\mathcal F$,
\eqref{eq:symbol-uniform} implies
\[
\left|a(x,\xi,\eta)\right|\lesssim\left\|a\right\|_{BS\mathcal A^m_\delta(r,\mu,v;N,\lambda)}
(1+\left|\xi\right|+\left|\eta\right|)^m
\]
for almost every $(x,\xi,\eta)\in\mathbb R^{3n}$.
Indeed, the embedding $B^N_{\lambda,\infty}\hookrightarrow L^\infty$
and \eqref{eq:symbol-uniform} give
$\left\|b_k(x,\cdot,\cdot)\right\|_{L^\infty(\mathbb R^{2n})}\lesssim
2^{km}\left\|a\right\|_{BS\mathcal A^m_\delta(r,\mu,v;N,\lambda)}$ for almost every $x$.
Since $\sum_{k\geq0}\chi_k=1$, we have
\[
a(x,\xi,\eta)=\sum_{k\geq0}b_k(x,2^{-k}\xi,2^{-k}\eta).
\]
At each $(\xi,\eta)$, only a uniformly bounded number of terms are
nonzero, and $2^k\asymp1+\left|\xi\right|+\left|\eta\right|$ for these terms, which proves
the estimate.
The associated operator is therefore defined on
$\mathcal S(\mathbb R^n)\times\mathcal S(\mathbb R^n)$ by
\begin{equation}\label{eq:operator}
T_a(f,g)(x)=\int_{\mathbb R^{2n}}
e^{2\pi ix\cdot(\xi+\eta)}a(x,\xi,\eta)
\widehat f(\xi)\widehat g(\eta)\,\mathrm d\xi\,\mathrm d\eta.
\end{equation}
The integral converges absolutely for almost every $x\in\mathbb R^n$.

\subsection{Main results}

\begin{theorem}\label{thm:triebel}
Let $p,p_1,p_2,q\in\mathcal P_0^{\log}$ satisfy
$p^-,p_1^-,p_2^-,q^->1$ and
\[
\frac1{p(x)}=\frac1{p_1(x)}+\frac1{p_2(x)},\qquad x\in\mathbb R^n.
\]
Let $s\in C^{\log}_{\rm loc}\cap L^\infty$, $m\in\mathbb R$,
$r>0$, $1\leq\mu\leq\infty$, $0<v\leq\infty$, $1\leq\lambda\leq2$,
$0\leq\delta<1$, and
\begin{equation}\label{eq:smoothness-range}
(1-\delta)r\geq\frac n\mu,
\qquad 0<s^-\leq s^+<r-\frac n\mu.
\end{equation}
Let $c_s,c_q$ be the constants fixed in Section~\ref{subsec:variable-spaces},
and let $N$ satisfy
\begin{equation}\label{eq:frequency-threshold}
N>\frac{2n}{\lambda}+2(n+c_s+c_q).
\end{equation}
Let $\mathcal A\in\{\mathcal B,\mathcal F\}$, with $\mu<\infty$
when $\mathcal A=\mathcal F$. If
$a\in BS\mathcal A^m_\delta(r,\mu,v;N,\lambda)$, then
\begin{equation}\label{eq:triebel-main}
\left\|T_a(f,g)\right\|_{F^{s(\cdot)}_{p(\cdot),q(\cdot)}}
\leq C\left\|a\right\|_{BS\mathcal A^m_\delta(r,\mu,v;N,\lambda)}\Bigl(
\left\|f\right\|_{F^{s(\cdot)+m}_{p_1(\cdot),q(\cdot)}}\left\|g\right\|_{L^{p_2(\cdot)}}+\left\|f\right\|_{L^{p_1(\cdot)}}
\left\|g\right\|_{F^{s(\cdot)+m}_{p_2(\cdot),q(\cdot)}}\Bigr)
\end{equation}
for all $f,g\in\mathcal S(\mathbb R^n)$. Moreover, one has
\begin{equation}\label{eq:besov-main}
\left\|T_a(f,g)\right\|_{B^{s(\cdot)}_{p(\cdot),q(\cdot)}}
\leq C\left\|a\right\|_{BS\mathcal A^m_\delta(r,\mu,v;N,\lambda)}\Bigl(
\left\|f\right\|_{B^{s(\cdot)+m}_{p_1(\cdot),q(\cdot)}}\left\|g\right\|_{L^{p_2(\cdot)}}+\left\|f\right\|_{L^{p_1(\cdot)}}
\left\|g\right\|_{B^{s(\cdot)+m}_{p_2(\cdot),q(\cdot)}}\Bigr)
\end{equation}
for all $f,g\in\mathcal S(\mathbb R^n)$.
The constant $C$ is independent of $a,f,g$.
\end{theorem}

To the best of the authors' knowledge, the estimates in
Theorem~\ref{thm:triebel} for the nonregular symbol classes considered here
are new even for the classical Besov and
Triebel--Lizorkin spaces with constant exponents.

\begin{corollary}\label{cor:extension}
Under the assumptions of Theorem~\ref{thm:triebel}, suppose in addition
that $s^-+m>0$. Then $T_a$ extends uniquely to continuous bilinear maps
\[
\begin{split}
F^{s(\cdot)+m}_{p_1(\cdot),q(\cdot)}
\times F^{s(\cdot)+m}_{p_2(\cdot),q(\cdot)}
&\longrightarrow F^{s(\cdot)}_{p(\cdot),q(\cdot)},\\
B^{s(\cdot)+m}_{p_1(\cdot),q(\cdot)}
\times B^{s(\cdot)+m}_{p_2(\cdot),q(\cdot)}
&\longrightarrow B^{s(\cdot)}_{p(\cdot),q(\cdot)}.
\end{split}
\]
The extensions satisfy \eqref{eq:triebel-main} and \eqref{eq:besov-main}.
\end{corollary}

Let $X$ be a quasi-Banach space of measurable functions on $\mathbb R^n$.
Its pointwise multiplier algebra $M(X)$ consists of the measurable
functions $h$ for which multiplication by $h$ maps $X$ boundedly into
itself, with
\[
\left\|h\right\|_{M(X)}
:=\sup_{\substack{f\in X\\\left\|f\right\|_X\leq1}}
\left\|hf\right\|_X<\infty.
\]
We call $X$ a multiplication algebra if $X\hookrightarrow M(X)$, that is,
if there exists $C>0$ such that
\[
\left\|fg\right\|_X\leq C\left\|f\right\|_X\left\|g\right\|_X,
\qquad f,g\in X.
\]
Classical multiplication algebra results for Besov and Triebel--Lizorkin
spaces can be found in \cite[Section~2.8.3, p.~145]{T83}.

For variable integrability, Naibo and Thomson
\cite[Section~6.3, Theorem~6.4]{NT19} proved Leibniz-type estimates
with an $L^\infty$ factor in Triebel--Lizorkin spaces based on
$L^{p(\cdot)}$. Hale and Naibo \cite[Section~7, Corollary~7.1]{HN23}
obtained corresponding estimates for both Besov and Triebel--Lizorkin
spaces. However, in these results the smoothness and summation exponents
are constants, so they do not cover the spaces considered here with
both $s(\cdot)$ and $q(\cdot)$ variable.
We apply the dyadic argument for $T_a$ to the multiplication operator
$T_1(f,g)=fg$ and obtain the following result.

\begin{theorem}\label{thm:algebra}
Let $p,q\in\mathcal P_0^{\log}$ with $p^-,q^->1$, and let
$s\in C^{\log}_{\rm loc}\cap L^\infty$ satisfy
\[
\operatorname*{ess\,inf}_{x\in\mathbb R^n}
\left(s(x)-\frac n{p(x)}\right)>0.
\]
Then $F^{s(\cdot)}_{p(\cdot),q(\cdot)}$ and
$B^{s(\cdot)}_{p(\cdot),q(\cdot)}$ are multiplication algebras.
More precisely, there exists $C>0$, independent of $f$ and $g$, such that
\begin{equation}\label{eq:algebra-triebel}
\left\|fg\right\|_{F^{s(\cdot)}_{p(\cdot),q(\cdot)}}
\leq C\left\|f\right\|_{F^{s(\cdot)}_{p(\cdot),q(\cdot)}}
\left\|g\right\|_{F^{s(\cdot)}_{p(\cdot),q(\cdot)}}
\end{equation}
for all $f,g\in F^{s(\cdot)}_{p(\cdot),q(\cdot)}$, and
\begin{equation}\label{eq:algebra-besov}
\left\|fg\right\|_{B^{s(\cdot)}_{p(\cdot),q(\cdot)}}
\leq C\left\|f\right\|_{B^{s(\cdot)}_{p(\cdot),q(\cdot)}}
\left\|g\right\|_{B^{s(\cdot)}_{p(\cdot),q(\cdot)}}
\end{equation}
for all $f,g\in B^{s(\cdot)}_{p(\cdot),q(\cdot)}$.
\end{theorem}

\section{Basic properties of the symbol classes}\label{sec:symbol-properties}

We study $BS\mathcal B^m_\delta(r,\mu,v;N,\lambda)$ in the range $1\leq\mu\leq\infty$
used in Theorem~\ref{thm:triebel}. Throughout this section,
$m\in\mathbb R$, $0\leq\delta<1$, $0<v\leq\infty$,
$1\leq\lambda\leq2$, $r>n/\mu$, and $N>2n/\lambda$,
unless stated otherwise. The constants may depend
on the fixed parameters and cutoffs, but not on the symbol.

\subsection{Dyadic resolutions and smooth multiplication}

We first record the Banach-valued estimates needed below.
Let $E$ be a complex Banach space and $d\in\mathbb N$.
For a strongly measurable function $h:\mathbb R^d\to E$, we use
\[
\left\|h\right\|_{L^u(\mathbb R^d;E)}
=\left(\int_{\mathbb R^d}\left\|h(x)\right\|_E^u\,\mathrm d x\right)^{1/u},
\qquad 1\leq u<\infty,
\]
and the essential supremum of $\left\|h(x)\right\|_E$ when $u=\infty$.
For $t>0$, $1\leq u\leq\infty$, and $0<w\leq\infty$, define
$B^t_{u,w}(\mathbb R^d;E)$ by
\[
\left\|b\right\|_{B^t_{u,w}(\mathbb R^d;E)}
=\left\|\left(2^{jt}\left\|\varphi^{(d)}_j*b\right\|_{L^u(\mathbb R^d;E)}
\right)_{j\in\mathbb N_0}\right\|_{\ell^w},
\]
where $b\in\mathcal S'(\mathbb R^d;E)$ and
$(\varphi^{(d)}_j)_{j\in\mathbb N_0}$ is a dyadic resolution in
$\mathbb R^d$ fixed in Section~\ref{sec:definitions}. These components
are smooth $E$-valued functions. The space consists of those $b$ for
which the displayed quasi-norm is finite.
For $d=n$ and $E=B^N_{\lambda,\infty}(\mathbb R^{2n})$, this
definition with $(t,u,w)=(r,\mu,v)$ agrees with
$B^r_{\mu,v}(B^N_{\lambda,\infty})$ in Definition~\ref{def:mixed-symbol}.

\Needspace{10\baselineskip}
\begin{lem}\label{lem:banach-besov}
Let $d\in\mathbb N$, $E$ be a complex Banach space, $t>0$, $1\leq u\leq\infty$, and
$0<w\leq\infty$.
\begin{enumerate}
\item[\textup{(i)}] The quasi-norm of $B^t_{u,w}(\mathbb R^d;E)$ is independent,
up to equivalence, of the smooth dyadic resolution.
\item[\textup{(ii)}] For every
$b\in B^t_{u,w}(\mathbb R^d;E)$ and $\zeta\in C_b^\infty(\mathbb R^d)$,
\[
\left\|\zeta b\right\|_{B^t_{u,w}(\mathbb R^d;E)}
\leq C_\zeta\left\|b\right\|_{B^t_{u,w}(\mathbb R^d;E)}.
\]
The constant $C_\zeta>0$ is independent of $b$;
only finitely many derivative bounds enter it.
\end{enumerate}
\end{lem}
\begin{proof}
\textup{(i)} Let $(\varphi^{(d)}_j)_{j\in\mathbb N_0}$ and
$(\widetilde\varphi^{(d)}_i)_{i\in\mathbb N_0}$ be two admissible
dyadic resolutions in $\mathbb R^d$. There exists an integer
$d_0\in\mathbb N_0$, depending only on their support radii, such that
\[
\widetilde\varphi^{(d)}_i*b
=\sum_{\substack{j\in\mathbb N_0\\\left|j-i\right|\leq d_0}}
\widetilde\varphi^{(d)}_i*(\varphi^{(d)}_j*b),
\qquad i\in\mathbb N_0.
\]
The kernels $\widetilde\varphi^{(d)}_i$ have uniformly bounded
$L^1(\mathbb R^d)$ norms. Since $u\geq1$, Young's inequality gives
\[
\left\|\widetilde\varphi^{(d)}_i*b\right\|_{L^u(\mathbb R^d;E)}
\lesssim\sum_{\substack{j\in\mathbb N_0\\\left|j-i\right|\leq d_0}}
\left\|\varphi^{(d)}_j*b\right\|_{L^u(\mathbb R^d;E)}.
\]
Since $2^{it}\asymp2^{jt}$ when $\left|i-j\right|\leq d_0$, the
$\min\{1,w\}$-triangle inequality in $\ell^w$ gives
\[
\begin{aligned}
&\left\|\left(2^{it}
\left\|\widetilde\varphi^{(d)}_i*b\right\|_{L^u(\mathbb R^d;E)}
\right)_{i\in\mathbb N_0}\right\|_{\ell^w}^{\min\{1,w\}}
\lesssim\sum_{h=-d_0}^{d_0}
\left\|\left(2^{(i+h)t}
\left\|\varphi^{(d)}_{i+h}*b\right\|_{L^u(\mathbb R^d;E)}
\right)_{i\in\mathbb N_0}\right\|_{\ell^w}^{\min\{1,w\}}\\
&\quad \leq(2d_0+1)
\left\|\left(2^{jt}
\left\|\varphi^{(d)}_j*b\right\|_{L^u(\mathbb R^d;E)}
\right)_{j\in\mathbb N_0}\right\|_{\ell^w}^{\min\{1,w\}}.
\end{aligned}
\]
Here $\varphi^{(d)}_j=0$ for $j<0$.
Interchanging the two resolutions gives the reverse inequality and
proves the equivalence of the quasi-norms.

\textup{(ii)} Choose $L\in\mathbb N$
with $L>t$. Taylor's formula and the vanishing moments of
$\varphi^{(d)}_i$, for $i\geq1$, give
\[
\left\|\varphi^{(d)}_i*\zeta\right\|_{L^\infty(\mathbb R^d)}
\leq C_L2^{-iL},\qquad i\in\mathbb N_0,
\]
where the bound at $i=0$ follows from $\left\|\zeta\right\|_{L^\infty(\mathbb R^d)}$.
The constant $C_L$ depends only on the dyadic kernels and
$\max_{\left|\beta\right|\leq L}\left\|\partial^\beta\zeta\right\|_{L^\infty(\mathbb R^d)}$,
where $\beta\in\mathbb N_0^d$.
Since $t>0$,
\[
\sum_{j=0}^{\infty}
\left\|\varphi^{(d)}_j*b\right\|_{L^u(\mathbb R^d;E)}
\leq\left(\sum_{j=0}^{\infty}2^{-jt}\right)
\left\|b\right\|_{B^t_{u,w}(\mathbb R^d;E)}<\infty.
\]
Also $\sum_{i=0}^{\infty}
\left\|\varphi^{(d)}_i*\zeta\right\|_{L^\infty(\mathbb R^d)}<\infty$.
Thus the double series for $\zeta b$ is absolutely convergent in
$L^u(\mathbb R^d;E)$ and may be split into
\[
\begin{aligned}
&\zeta b
=\sum_{j=4}^{\infty}
\left(\sum_{i=0}^{j-4}\varphi^{(d)}_i*\zeta\right)
(\varphi^{(d)}_j*b)
+\sum_{i=4}^{\infty}(\varphi^{(d)}_i*\zeta)
\left(\sum_{j=0}^{i-4}\varphi^{(d)}_j*b\right)\\
&\quad +\sum_{j=0}^{\infty}
\sum_{\substack{i\in\mathbb N_0\\\left|i-j\right|\leq3}}
(\varphi^{(d)}_i*\zeta)(\varphi^{(d)}_j*b)
:=S_1+S_2+S_3.
\end{aligned}
\]
There exist fixed integers $d_1,d_2\in\mathbb N_0$, depending only
on the Fourier supports of the dyadic kernels, with the following
properties. In $S_1$, the term indexed by $j$ contributes to
$\varphi^{(d)}_h*S_1$ only when $\left|h-j\right|\leq d_1$; in $S_2$, the
term indexed by $i$ contributes only when $\left|h-i\right|\leq d_1$.
In $S_3$, the product indexed by $(i,j)$ contributes only when
$h\leq j+d_2$. Here and below $h\in\mathbb N_0$.

For $S_1$, the low-frequency partial sums satisfy
\[
\sup_{j\geq4}
\left\|\sum_{i=0}^{j-4}\varphi^{(d)}_i*\zeta
\right\|_{L^\infty(\mathbb R^d)}
\lesssim\left\|\zeta\right\|_{L^\infty(\mathbb R^d)}.
\]
Young's inequality and $2^{ht}\asymp2^{jt}$ for $\left|h-j\right|\leq d_1$
therefore give
\[
\begin{aligned}
2^{ht}\left\|\varphi^{(d)}_h*S_1\right\|_{L^u(\mathbb R^d;E)}
&\lesssim\left\|\zeta\right\|_{L^\infty(\mathbb R^d)}
\sum_{\substack{j\geq4\\\left|h-j\right|\leq d_1}}
2^{jt}\left\|\varphi^{(d)}_j*b\right\|_{L^u(\mathbb R^d;E)}.
\end{aligned}
\]
Applying the same finite-shift estimate as above yields
\[
\left\|S_1\right\|_{B^t_{u,w}(\mathbb R^d;E)}
\lesssim\left\|\zeta\right\|_{L^\infty(\mathbb R^d)}
\left\|\left(2^{jt}\left\|\varphi^{(d)}_j*b\right\|_{L^u(\mathbb R^d;E)}
\right)_{j\in\mathbb N_0}\right\|_{\ell^w}
=\left\|\zeta\right\|_{L^\infty(\mathbb R^d)}
\left\|b\right\|_{B^t_{u,w}(\mathbb R^d;E)}.
\]

For $S_2$, using the decay of $\varphi^{(d)}_i*\zeta$, we obtain
\[
2^{ht}\left\|\varphi^{(d)}_h*S_2\right\|_{L^u(\mathbb R^d;E)}
\lesssim C_L
\sum_{\substack{i\geq4\\\left|h-i\right|\leq d_1}}
2^{ht-iL}\sum_{j=0}^{i-4}
\left\|\varphi^{(d)}_j*b\right\|_{L^u(\mathbb R^d;E)}
\lesssim C_L2^{-h(L-t)}
\left(\sum_{j=0}^{\infty}2^{-jt}\right)
\left\|b\right\|_{B^t_{u,w}(\mathbb R^d;E)}.
\]
Since $L>t$, taking the $\ell^w$ quasi-norm gives
\[
\left\|S_2\right\|_{B^t_{u,w}(\mathbb R^d;E)}
\lesssim C_L
\left\|(2^{-h(L-t)})_{h\in\mathbb N_0}\right\|_{\ell^w}
\left\|b\right\|_{B^t_{u,w}(\mathbb R^d;E)}
\lesssim C_L\left\|b\right\|_{B^t_{u,w}(\mathbb R^d;E)}.
\]

For $S_3$, Young's inequality and the uniformly bounded number of
indices $i$ with $\left|i-j\right|\leq3$ imply
\[
\begin{aligned}
2^{ht}\left\|\varphi^{(d)}_h*S_3\right\|_{L^u(\mathbb R^d;E)}
&\lesssim C_L
\sum_{j=\max\left\{0,h-d_2\right\}}^{\infty}
2^{ht-jL}\left\|\varphi^{(d)}_j*b\right\|_{L^u(\mathbb R^d;E)}\\
&\leq C_L
\sum_{j=\max\left\{0,h-d_2\right\}}^{\infty}
2^{-(j-h)t}\,2^{jt}
\left\|\varphi^{(d)}_j*b\right\|_{L^u(\mathbb R^d;E)}.
\end{aligned}
\]
Applying the $\min\{1,w\}$-triangle inequality to the shifts
$k=j-h\geq-d_2$ gives
\[
\left\|S_3\right\|_{B^t_{u,w}(\mathbb R^d;E)}^{\min\{1,w\}}
\lesssim C_L^{\min\{1,w\}}
\left(\sum_{k=-d_2}^{\infty}2^{-kt\min\{1,w\}}\right)
\left\|b\right\|_{B^t_{u,w}(\mathbb R^d;E)}^{\min\{1,w\}}
\lesssim C_L^{\min\{1,w\}}
\left\|b\right\|_{B^t_{u,w}(\mathbb R^d;E)}^{\min\{1,w\}}.
\]
Combining the three
estimates with the $\min\{1,w\}$-triangle inequality gives
\[
\left\|\zeta b\right\|_{B^t_{u,w}(\mathbb R^d;E)}^{\min\{1,w\}}
\leq\sum_{\nu=1}^3
\left\|S_\nu\right\|_{B^t_{u,w}(\mathbb R^d;E)}^{\min\{1,w\}}
\lesssim C_L^{\min\{1,w\}}
\left\|b\right\|_{B^t_{u,w}(\mathbb R^d;E)}^{\min\{1,w\}},
\]
which proves the multiplication estimate.
\end{proof}

Now, we show that the class $BS\mathcal B^m_\delta(r,\mu,v;N,\lambda)$
is independent of the decompositions.

\begin{proposition}\label{prop:symbol-independent}
Let $m\in\mathbb R$, $0\leq\delta<1$, $1\leq\mu\leq\infty$,
$0<v\leq\infty$, $1\leq\lambda\leq2$, $r>n/\mu$, and $N>2n/\lambda$.
The class $BS\mathcal B^m_\delta(r,\mu,v;N,\lambda)$ is unchanged,
with equivalent quasi-norms, under each of the following changes:
\begin{enumerate}
\item[\textup{(i)}] the spatial dyadic resolution;
\item[\textup{(ii)}] the dyadic resolution defining
$B^N_{\lambda,\infty}(\mathbb R^{2n})$;
\item[\textup{(iii)}] the joint-frequency cutoff $\Psi$ used to form the
functions $b_k$ in \eqref{eq:profile}.
\end{enumerate}
In \textup{(iii)}, the replacement
is a function $\widetilde\Psi\in C_c^\infty(\mathbb R^{2n})$ equal
to one near the origin, giving dyadic cutoffs supported in
$\left\{(\xi,\eta)\in\mathbb R^{2n}:c_1 2^k\leq\left|(\xi,\eta)\right|\leq c_2 2^k\right\}$ for $k\geq1$,
where $0<c_1<c_2<\infty$ are independent of $k$.
\end{proposition}
\begin{proof}
Put $E=B^N_{\lambda,\infty}(\mathbb R^{2n})$, which is a Banach
space since $\lambda\geq1$. Part \textup{(i)} follows from
Lemma~\ref{lem:banach-besov}.

Part \textup{(ii)} follows from the classical independence of the
$B^N_{\lambda,\infty}(\mathbb R^{2n})$ norm, up to equivalence,
from the dyadic resolution.

For \textup{(iii)}, denote the new joint cutoffs and normalized functions by
$\widetilde\chi_k$ and $\widetilde b_k$. The partition identity gives
\[
\widetilde b_k(x,\xi,\eta)
=\widetilde\chi_k(2^k\xi,2^k\eta)
\sum_{\substack{h\in\mathbb N_0\\\left|h-k\right|\leq d_2}}
b_h(x,2^{k-h}\xi,2^{k-h}\eta),
\]
for $k\in\mathbb N_0$, where $d_2\in\mathbb N_0$ is chosen
large enough to cover all overlaps, including the low-frequency cutoffs.
The normalized cutoffs $\widetilde\chi_k(2^k\cdot)$ have fixed
compact supports and uniformly bounded derivatives. The dilations
$2^{k-h}$ range over a fixed finite set when $\left|h-k\right|\leq d_2$.
The boundedness of these dyadic dilations and
Lemma~\ref{lem:banach-besov}, applied in dimension $2n$, give
\[
\left\|\widetilde\chi_k(2^k\cdot)\,
 u(2^{k-h}\cdot)\right\|_{B^N_{\lambda,\infty}(\mathbb R^{2n})}
\lesssim\left\|u\right\|_{B^N_{\lambda,\infty}(\mathbb R^{2n})},
\]
for every $u\in E$ and $k,h\in\mathbb N_0$ with $\left|h-k\right|\leq d_2$,
with a constant independent of $u$, $k$, and $h$.
Consequently,
\[
\left\|\widetilde b_k\right\|_{L^\infty(\mathbb R^n;E)}
\lesssim\sum_{\substack{h\in\mathbb N_0\\\left|h-k\right|\leq d_2}}
\left\|b_h\right\|_{L^\infty(\mathbb R^n;E)}.
\]
Since the frequency multiplications and dilations commute with
spatial convolution, for $\ell,k\in\mathbb N_0$ and almost every
$x\in\mathbb R^n$ we also have
\[
\left\| (\varphi_\ell*_x\widetilde b_k)(x,\cdot,\cdot)\right\|_E
\lesssim\sum_{\substack{h\in\mathbb N_0\\\left|h-k\right|\leq d_2}}
\left\| (\varphi_\ell*_x b_h)(x,\cdot,\cdot)\right\|_E.
\]
Taking the $L^\mu$ norm and then using the
$\min\{1,v\}$-triangle inequality in $\ell^v$ yields
\[
\left\|\widetilde b_k\right\|_{B^r_{\mu,v}(\mathbb R^n;E)}^{\min\{1,v\}}
\lesssim\sum_{\substack{h\in\mathbb N_0\\\left|h-k\right|\leq d_2}}
\left\|b_h\right\|_{B^r_{\mu,v}(\mathbb R^n;E)}^{\min\{1,v\}}.
\]
Since $2^{km}\asymp2^{hm}$ and
$2^{k(m+\delta r)}\asymp2^{h(m+\delta r)}$ when $\left|h-k\right|\leq d_2$,
the preceding bounds imply
\[
\begin{aligned}
\sup_{k\in\mathbb N_0}2^{-km}
\left\|\widetilde b_k\right\|_{L^\infty(\mathbb R^n;E)}
&\lesssim\sup_{h\in\mathbb N_0}2^{-hm}
\left\|b_h\right\|_{L^\infty(\mathbb R^n;E)},\\[2mm]
\sup_{k\in\mathbb N_0}2^{-k(m+\delta r)}
\left\|\widetilde b_k\right\|_{B^r_{\mu,v}(\mathbb R^n;E)}
&\lesssim\sup_{h\in\mathbb N_0}2^{-h(m+\delta r)}
\left\|b_h\right\|_{B^r_{\mu,v}(\mathbb R^n;E)}.
\end{aligned}
\]
These are the two required symbol estimates. Interchanging the
joint-frequency decompositions gives the reverse bounds.
\end{proof}

\subsection{Embeddings}

\begin{proposition}\label{prop:symbol-embeddings}
Let $1\leq\mu\leq\infty$, $1\leq\lambda\leq2$,
$m,m_i\in\mathbb R$, $\delta,\delta_i\in[0,1)$,
$v,v_i\in(0,\infty]$, $r,r_i>n/\mu$, and
$N,N_i>2n/\lambda$, for $i=1,2$.
In \textup{(i)} and \textup{(ii)}, assume also that
$N_1\geq N_2$ and $m_1\leq m_2$.
\begin{enumerate}
\item[\textup{(i)}] If $r_1>r_2$ and
$m_1+\delta_1r_1\leq m_2+\delta_2r_2$, then
\[
BS\mathcal B^{m_1}_{\delta_1}(r_1,\mu,v_1;N_1,\lambda)
\hookrightarrow
BS\mathcal B^{m_2}_{\delta_2}(r_2,\mu,v_2;N_2,\lambda).
\]
\item[\textup{(ii)}] If $v_1\leq v_2$ and
$m_1+\delta_1r\leq m_2+\delta_2r$, then
\[
BS\mathcal B^{m_1}_{\delta_1}(r,\mu,v_1;N_1,\lambda)
\hookrightarrow
BS\mathcal B^{m_2}_{\delta_2}(r,\mu,v_2;N_2,\lambda).
\]
\item[\textup{(iii)}] If $\mu<\infty$, then
\[
BS\mathcal B^m_\delta(r,\mu,\min\{\mu,v\};N,\lambda)
\hookrightarrow BS\mathcal F^m_\delta(r,\mu,v;N,\lambda)
\hookrightarrow
BS\mathcal B^m_\delta(r,\mu,\max\{\mu,v\};N,\lambda).
\]
\end{enumerate}
\end{proposition}
\begin{proof}
\textup{(i)} For $u\in B^{N_1}_{\lambda,\infty}(\mathbb R^{2n})$, we have
\[
\left\|u\right\|_{B^{N_2}_{\lambda,\infty}(\mathbb R^{2n})}
=\sup_{j\in\mathbb N_0}2^{-j(N_1-N_2)}2^{jN_1}
\left\|\varphi^{(2n)}_j*u\right\|_{L^\lambda(\mathbb R^{2n})}
\leq\left\|u\right\|_{B^{N_1}_{\lambda,\infty}(\mathbb R^{2n})}.
\]
Let $a\in BS\mathcal B^{m_1}_{\delta_1}(r_1,\mu,v_1;N_1,\lambda)$,
and let $b_k$, $k\in\mathbb N_0$, be defined by \eqref{eq:profile}.
Since $r_1>r_2$, we have
\[
\begin{aligned}
&\left\|b_k\right\|_{B^{r_2}_{\mu,v_2}(B^{N_2}_{\lambda,\infty})}
\leq
\left\|(2^{-\ell(r_1-r_2)})_{\ell\in\mathbb N_0}
\right\|_{\ell^{v_2}}
\sup_{\ell\in\mathbb N_0}2^{\ell r_1}
\left\|\varphi_\ell*_x b_k\right\|_{L^\mu(\mathbb R^n;
 B^{N_1}_{\lambda,\infty}(\mathbb R^{2n}))}\\
&\quad \leq
\left\|(2^{-\ell(r_1-r_2)})_{\ell\in\mathbb N_0}
\right\|_{\ell^{v_2}}
\left\|b_k\right\|_{B^{r_1}_{\mu,v_1}(B^{N_1}_{\lambda,\infty})}
\lesssim\left\|b_k\right\|_{B^{r_1}_{\mu,v_1}(B^{N_1}_{\lambda,\infty})}.
\end{aligned}
\]
The constants are independent of $k$. Multiplying the uniform
frequency estimate by $2^{-km_2}$ and the spatial estimate by
$2^{-k(m_2+\delta_2r_2)}$, and using
$m_1\leq m_2$ and $m_1+\delta_1r_1\leq m_2+\delta_2r_2$, yields
\[
\left\|a\right\|_{BS\mathcal B^{m_2}_{\delta_2}(r_2,\mu,v_2;N_2,\lambda)}
\lesssim
\left\|a\right\|_{BS\mathcal B^{m_1}_{\delta_1}(r_1,\mu,v_1;N_1,\lambda)}.
\]
This proves \textup{(i)}.

\textup{(ii)} Let $a\in BS\mathcal B^{m_1}_{\delta_1}(r,\mu,v_1;N_1,\lambda)$,
and let $b_k$, $k\in\mathbb N_0$, be defined by \eqref{eq:profile}.
The frequency estimate in \textup{(i)} and the inclusion
$\ell^{v_1}\hookrightarrow\ell^{v_2}$ give
\[
\left\|b_k\right\|_{B^r_{\mu,v_2}(B^{N_2}_{\lambda,\infty})}
\leq\left\|\left(2^{\ell r}
\left\|\varphi_\ell*_x b_k\right\|_{L^\mu(\mathbb R^n;
 B^{N_1}_{\lambda,\infty}(\mathbb R^{2n}))}
\right)_{\ell\in\mathbb N_0}\right\|_{\ell^{v_2}}
\leq\left\|b_k\right\|_{B^r_{\mu,v_1}(B^{N_1}_{\lambda,\infty})}.
\]
Multiplying the frequency and spatial estimates by $2^{-km_2}$ and
$2^{-k(m_2+\delta_2r)}$, respectively, and using the assumptions on
$m_i$ and $\delta_i$, we obtain
\[
\left\|a\right\|_{BS\mathcal B^{m_2}_{\delta_2}(r,\mu,v_2;N_2,\lambda)}
\lesssim
\left\|a\right\|_{BS\mathcal B^{m_1}_{\delta_1}(r,\mu,v_1;N_1,\lambda)},
\]
which proves \textup{(ii)}.

\textup{(iii)} The two embeddings follow from the mixed-norm inequalities
underlying the classical Besov--Triebel--Lizorkin embeddings;
see \cite[p.~47, Proposition~2]{T83}.
\end{proof}

\subsection{Cutoff stability}

\begin{proposition}\label{prop:symbol-cutoffs}
Let $m\in\mathbb R$, $0\leq\delta<1$, $1\leq\mu\leq\infty$,
$0<v\leq\infty$, $1\leq\lambda\leq2$, $r>n/\mu$, and $N>2n/\lambda$.
Let $a\in BS\mathcal B^m_\delta(r,\mu,v;N,\lambda)$.
If $\zeta\in C_b^\infty(\mathbb R^n)$ and
$\Theta\in C_c^\infty(\mathbb R^{2n})$, then, for every integer
$J\in\mathbb N_0$,
\[
a_J(x,\xi,\eta)=\zeta(x)\Theta(2^{-J}\xi,2^{-J}\eta)a(x,\xi,\eta)
\]
belongs to $BS\mathcal B^m_\delta(r,\mu,v;N,\lambda)$ and
\[
\left\|a_J\right\|_{BS\mathcal B^m_\delta(r,\mu,v;N,\lambda)}\leq C_{\zeta,\Theta}\left\|a\right\|_{BS\mathcal B^m_\delta(r,\mu,v;N,\lambda)},
\]
where $C_{\zeta,\Theta}>0$ depends on the cutoffs and the fixed
parameters, but is independent of $J$ and $a$. Multiplication by
$\zeta(x)$ alone also preserves the class.
\end{proposition}
\begin{proof}
Put $E=B^N_{\lambda,\infty}(\mathbb R^{2n})$. For $k\in\mathbb N_0$,
the normalized function associated with $a_J$ is
\[
\zeta(x)\Theta(2^{k-J}\xi,2^{k-J}\eta)b_k(x,\xi,\eta).
\]
Choose $M\in\mathbb N$ and $\Omega\in C_c^\infty(\mathbb R^{2n})$
such that $\operatorname{supp}\Theta\subset B(0,2^{M-1})$ and
$\Omega=1$ on $B(0,2)$. The support properties of $b_k$ give
\[
\Theta(2^{k-J}\cdot)b_k
=\Omega\,\Theta(2^{k-J}\cdot)b_k,\qquad
\Theta(2^{k-J}\cdot)b_k=0\quad\text{if }k>J+M.
\]
For every $\gamma\in\mathbb N_0^{2n}$, the remaining multipliers satisfy
\[
\sup_{\substack{k,J\in\mathbb N_0\\k\leq J+M}}
\left\|\partial^\gamma\left[\Omega\,\Theta(2^{k-J}\cdot)\right]
\right\|_{L^\infty(\mathbb R^{2n})}\leq C_{\gamma,\Theta}<\infty.
\]
Lemma~\ref{lem:banach-besov}, applied in dimension $2n$, therefore gives
\[
\left\|\Theta(2^{k-J}\cdot)b_k(x,\cdot,\cdot)\right\|_E
\leq C_\Theta\left\|b_k(x,\cdot,\cdot)\right\|_E,
\]
for all $k,J\in\mathbb N_0$ and almost every $x\in\mathbb R^n$,
with $C_\Theta>0$ independent of $k$, $J$, and $x$.
Since the frequency multiplication commutes with spatial convolution,
taking the $L^\mu$ norm and the $\ell^v$ quasi-norm yields
\[
\left\|\Theta(2^{k-J}\cdot)b_k\right\|_{B^r_{\mu,v}(\mathbb R^n;E)}
\leq C_\Theta\left\|b_k\right\|_{B^r_{\mu,v}(\mathbb R^n;E)}.
\]
Applying Lemma~\ref{lem:banach-besov} in the spatial variable and
using the $L^\infty$ bound for $\zeta$, we obtain
\[
\begin{aligned}
\left\|\zeta\,\Theta(2^{k-J}\cdot)b_k\right\|_{L^\infty(\mathbb R^n;E)}
&\leq C_\Theta\left\|\zeta\right\|_{L^\infty(\mathbb R^n)}
\left\|b_k\right\|_{L^\infty(\mathbb R^n;E)},\\
\left\|\zeta\,\Theta(2^{k-J}\cdot)b_k\right\|_{B^r_{\mu,v}(\mathbb R^n;E)}
&\leq C_\zeta C_\Theta\left\|b_k\right\|_{B^r_{\mu,v}(\mathbb R^n;E)}.
\end{aligned}
\]
Multiplying these inequalities by $2^{-km}$ and
$2^{-k(m+\delta r)}$, respectively, and taking the suprema over
$k\in\mathbb N_0$, gives
\[
\left\|a_J\right\|_{BS\mathcal B^m_\delta(r,\mu,v;N,\lambda)}
\leq C_{\zeta,\Theta}\left\|a\right\|_{BS\mathcal B^m_\delta(r,\mu,v;N,\lambda)}.
\]
The assertion for multiplication by $\zeta(x)$ alone follows
directly from Lemma~\ref{lem:banach-besov}.
\end{proof}

\subsection{Relation to classical smooth symbols}

Let $m\in\mathbb R$ and $0\leq\delta\leq1$. Following
\cite[Section~2]{BMNT10}, the classical bilinear H\"ormander class
$BS^m_{1,\delta}$ consists of all $a\in C^\infty(\mathbb R^{3n})$
such that
\begin{equation}\label{eq:hormander}
\left|\partial_x^\gamma\partial_\xi^\alpha\partial_\eta^\beta
 a(x,\xi,\eta)\right|
\leq C_{\gamma,\alpha,\beta}
(1+\left|\xi\right|+\left|\eta\right|)^{m+\delta\left|\gamma\right|-\left|\alpha\right|-\left|\beta\right|}
\end{equation}
for all $(x,\xi,\eta)\in\mathbb R^{3n}$ and all
$\gamma,\alpha,\beta\in\mathbb N_0^n$. Here
$\left|\gamma\right|=\gamma_1+\cdots+\gamma_n$, and similarly for $\alpha$ and $\beta$;
each $C_{\gamma,\alpha,\beta}>0$ is independent of $x$, $\xi$, and $\eta$.
For $a\in BS^m_{1,\delta}$, the associated classical smooth bilinear
pseudodifferential operator is $T_a$ defined by \eqref{eq:operator}.
The early study of these operators goes back to Coifman and Meyer
\cite{CM75,CM78}.

First, we have the following inclusion.

\begin{proposition}\label{prop:hormander}
Let $m\in\mathbb R$, $0\leq\delta<1$, $r>0$, $1\leq\lambda\leq2$,
and $N>2n/\lambda$. Then
\[
BS^m_{1,\delta}\subset
BS\mathcal B^m_\delta(r,\infty,\infty;N,\lambda).
\]
\end{proposition}
\begin{proof}
Let $a\in BS^m_{1,\delta}$, and let $b_k$, $k\in\mathbb N_0$,
be defined by \eqref{eq:profile}.
On the support of $\chi_k$, $1+\left|\xi\right|+\left|\eta\right|\asymp2^k$. For every $\gamma,\alpha,\beta\in\mathbb N_0^n$ and $k\in\mathbb N_0$,
\eqref{eq:hormander} and Leibniz's formula, together with the uniformly
bounded derivatives of $\chi_k(2^k\cdot)$, give
\[
\sup_{(x,\xi,\eta)\in\mathbb R^{3n}}
\left|\partial_x^\gamma\partial_\xi^\alpha\partial_\eta^\beta
 b_k(x,\xi,\eta)\right|
\lesssim\sum_{\substack{\alpha_1\leq\alpha\\\beta_1\leq\beta}}
2^{k(\left|\alpha_1\right|+\left|\beta_1\right|)}
2^{k(m+\delta\left|\gamma\right|-\left|\alpha_1\right|-\left|\beta_1\right|)}
\lesssim2^{k(m+\delta\left|\gamma\right|)}.
\]
Here $\alpha_1\leq\alpha$ and $\beta_1\leq\beta$ mean
componentwise inequalities. Choose $K\in\mathbb N$ with $K>N$.
Since the normalized functions have fixed compact support,
Taylor's formula and the vanishing moments of the frequency kernels give
\[
\sup_{x\in\mathbb R^n}
\left\|\varphi_j^{(2n)}*_{(\xi,\eta)}
\partial_x^\gamma b_k(x,\cdot,\cdot)\right\|_{L^\lambda(\mathbb R^{2n})}
\lesssim2^{-jK}2^{k(m+\delta\left|\gamma\right|)},\qquad j,k\in\mathbb N_0.
\]
The estimate for $j=0$ follows from Young's inequality. Hence
\[
\sup_{x\in\mathbb R^n}
\left\|\partial_x^\gamma b_k(x,\cdot,\cdot)\right\|_{B^N_{\lambda,\infty}(\mathbb R^{2n})}
\lesssim2^{k(m+\delta\left|\gamma\right|)}
\sup_{j\in\mathbb N_0}2^{-j(K-N)}
=2^{k(m+\delta\left|\gamma\right|)}.
\]
Taking $\gamma=0$ proves \eqref{eq:symbol-uniform}.

Choose $L\in\mathbb N$ with $L>r$. The cancellation of the spatial dyadic
kernels for $\ell\geq1$ gives
\[
\sup_{x\in\mathbb R^n}\left\|(\varphi_\ell*_x b_k)(x,\cdot,\cdot)
\right\|_{B^N_{\lambda,\infty}}
\lesssim2^{km}\min\left\{1,2^{-L(\ell-k\delta)}\right\}.
\]
For $\ell=0$ the uniform bound suffices. Splitting at $\ell=k\delta$
shows that
\[
\sup_{\ell\geq0}2^{\ell r}
\sup_{x\in\mathbb R^n}\left\|(\varphi_\ell*_x b_k)(x,\cdot,\cdot)
\right\|_{B^N_{\lambda,\infty}}
\lesssim2^{k(m+\delta r)},
\]
which proves \eqref{eq:symbol-spatial} with $\mu=v=\infty$.
\end{proof}

\Needspace{4\baselineskip}
The following example shows that symbols in our classes need not be
infinitely differentiable in the spatial variable.

\begin{example}\label{ex:rough-coefficient}
Let $m\in\mathbb R$, $1\leq\mu\leq\infty$, $0<v\leq\infty$,
$r>n/\mu$, $1\leq\lambda\leq2$, and $N>2n/\lambda$.
Let $b\in B^r_{\mu,v}(\mathbb R^n)\cap L^\infty(\mathbb R^n)$. Then
\[
a(x,\xi,\eta)=b(x)(1+\left|\xi\right|^2+\left|\eta\right|^2)^{m/2}
\]
belongs to $BS\mathcal B^m_0(r,\mu,v;N,\lambda)$.

Indeed, for $k\in\mathbb N_0$, the normalized function is
\[
b_k(x,\xi,\eta)=2^{km}b(x)
(2^{-2k}+\left|\xi\right|^2+\left|\eta\right|^2)^{m/2}\chi_k(2^k\xi,2^k\eta).
\]
The frequency calculation in the proof of
Proposition~\ref{prop:hormander} gives the uniform bound
\[
\left\|(2^{-2k}+\left|\xi\right|^2+\left|\eta\right|^2)^{m/2}
\chi_k(2^k\xi,2^k\eta)\right\|_{B^N_{\lambda,\infty}(\mathbb R^{2n})}
\lesssim1.
\]
Since the frequency factor is independent of $x$, we have
\[
(\varphi_\ell*_x b_k)(x,\xi,\eta)=2^{km}(\varphi_\ell*b)(x)
(2^{-2k}+\left|\xi\right|^2+\left|\eta\right|^2)^{m/2}\chi_k(2^k\xi,2^k\eta),
\]
for $\ell,k\in\mathbb N_0$. Consequently,
\[
\begin{aligned}
\mathop{\rm ess\,sup}_{x\in\mathbb R^n}
\left\|b_k(x,\cdot,\cdot)\right\|_{B^N_{\lambda,\infty}(\mathbb R^{2n})}
&\lesssim2^{km}\left\|b\right\|_{L^\infty(\mathbb R^n)},\\
\left\|b_k\right\|_{B^r_{\mu,v}(B^N_{\lambda,\infty})}
&\lesssim2^{km}\left\|b\right\|_{B^r_{\mu,v}(\mathbb R^n)}.
\end{aligned}
\]
Taking the two weighted suprema yields
\[
\left\|a\right\|_{BS\mathcal B^m_0(r,\mu,v;N,\lambda)}
\lesssim\left\|b\right\|_{L^\infty(\mathbb R^n)}
+\left\|b\right\|_{B^r_{\mu,v}(\mathbb R^n)}.
\]
If $1\leq\mu<\infty$ and
$b\in F^r_{\mu,v}(\mathbb R^n)\cap L^\infty(\mathbb R^n)$,
the same calculation, with the $\ell^v$ quasi-norm inside $L^\mu$, gives
\[
\left\|a\right\|_{BS\mathcal F^m_0(r,\mu,v;N,\lambda)}
\lesssim\left\|b\right\|_{L^\infty(\mathbb R^n)}
+\left\|b\right\|_{F^r_{\mu,v}(\mathbb R^n)}.
\]
In particular,
\[
b\notin C^\infty(\mathbb R^n)
\quad\Longrightarrow\quad
 a\in BS\mathcal B^m_0(r,\mu,v;N,\lambda)
 \setminus C^\infty(\mathbb R^{3n}),
\]
since $a(\cdot,0,0)=b$.
\end{example}

\Needspace{4\baselineskip}
The next example gives a symbol that is not infinitely differentiable
in either the spatial variable or the joint frequency variables.

\begin{example}\label{ex:rough-both}
Let $m\in\mathbb R$, $0\leq\delta<1$, $1\leq\mu\leq\infty$,
$0<v\leq\infty$, $r>n/\mu$, $1\leq\lambda\leq2$, and
$N>2n/\lambda$. Choose $L_x,L_{\xi,\eta}\in\mathbb N$ with
$L_x>r$ and $L_{\xi,\eta}>N$, and noninteger real numbers
$\alpha>L_x$ and $\beta>L_{\xi,\eta}$.
Let $\zeta\in C_c^\infty(\mathbb R^n)$ be equal to one near the
origin, and put
\[
b(x)=\zeta(x)\left|x_1\right|^\alpha,\qquad x\in\mathbb R^n.
\]
Let $e_1=(1,0,\ldots,0)\in\mathbb R^n$.
Choose $\Theta\in C_c^\infty(\mathbb R^{2n})$ equal to one near
$(0,e_1)$ and supported in $B((0,e_1),1/16)$, and define
\[
\tau(\xi,\eta)=\Theta(\xi,\eta)\left|\xi_1\right|^\beta,
\qquad (\xi,\eta)\in\mathbb R^{2n}.
\]
Then $b\in C_c^{L_x}(\mathbb R^n)$ and
$\tau\in C_c^{L_{\xi,\eta}}(\mathbb R^{2n})$, but neither function is smooth
of every order. The symbol
\begin{equation}\label{eq:rough-dyadic-symbol}
a(x,\xi,\eta)=\sum_{j=1}^{\infty}2^{jm}
b(2^{j\delta}x)\tau(2^{-j}\xi,2^{-j}\eta)
\end{equation}
belongs to $BS\mathcal B^m_\delta(r,\mu,v;N,\lambda)$.
However, $a$ is neither $C^\infty$ in the spatial variable $x$ nor
$C^\infty$ in the joint frequency variables $(\xi,\eta)$.

To verify the assertion for \eqref{eq:rough-dyadic-symbol}, note that
$\operatorname{supp}\Theta\subset B((0,e_1),1/16)$ implies
\[
\operatorname{supp}[\tau(2^{-j}\cdot)]
\subset\left\{(\xi,\eta)\in\mathbb R^{2n}:
\frac{15}{16}2^j\leq\left|(\xi,\eta)\right|\leq\frac{17}{16}2^j\right\},
\qquad j\in\mathbb N.
\]
Since
\[
\frac{17}{16}2^j<\frac{15}{16}2^{j+1},\qquad j\in\mathbb N,
\]
these supports are pairwise disjoint. In particular,
\[
\sum_{j=1}^{\infty}
\mathbf 1_{\operatorname{supp}[\tau(2^{-j}\cdot)]}(\xi,\eta)
\leq1,\qquad (\xi,\eta)\in\mathbb R^{2n}.
\]
For $t\geq1$, cancellation of the spatial dyadic kernels and
Young's inequality give
\[
\left\|\varphi_\ell*[b(t\cdot)]\right\|_{L^\mu(\mathbb R^n)}
\lesssim t^{-n/\mu}\min\left\{1,2^{-\ell L_x}t^{L_x}\right\},
\qquad \ell\geq1,
\]
and the first bound $Ct^{-n/\mu}$ suffices for $\ell=0$.
The constants are independent of $t$ and $\ell$; they depend on the
$L^\mu$ norms of $b$ and its derivatives of order $L_x$.
Splitting the sequence at $\ell=\lfloor\log_2 t\rfloor$ yields
\[
\left\|b(t\cdot)\right\|_{B^r_{\mu,v}}
\lesssim t^{r-n/\mu}\leq t^r,
\qquad t\geq1.
\]
Similarly, the bounded compactly supported derivatives of $\tau$
of order $L_{\xi,\eta}>N$ imply
$\tau\in B^N_{\lambda,\infty}(\mathbb R^{2n})$.

For $k\in\mathbb N_0$, the normalized function $b_k$ in
\eqref{eq:profile} has the form
\[
b_k(x,\xi,\eta)=\chi_k(2^k\xi,2^k\eta)
\sum_{\substack{j\in\mathbb N\\\left|j-k\right|\leq d_0}}
2^{jm}b(2^{j\delta}x)\tau(2^{k-j}\xi,2^{k-j}\eta),
\]
for a fixed $d_0\in\mathbb N_0$ covering all support overlaps,
including $k=0$. Smooth multiplication and the finitely many
dilations $2^{k-j}$ are bounded on
$B^N_{\lambda,\infty}(\mathbb R^{2n})$.
Consequently,
\[
\begin{split}
\mathop{\rm ess\,sup}_{x\in\mathbb R^n}
\left\|b_k(x,\cdot,\cdot)\right\|_{B^N_{\lambda,\infty}}
&\lesssim2^{km},\\
\left\|b_k\right\|_{B^r_{\mu,v}(B^N_{\lambda,\infty})}
&\lesssim\sum_{\substack{j\in\mathbb N\\\left|j-k\right|\leq d_0}}
2^{jm}2^{j\delta r}
\lesssim2^{k(m+\delta r)}.
\end{split}
\]
This proves that $a$ belongs to the stated symbol class.
\end{example}

Finally, we show that the inclusion in Proposition~\ref{prop:hormander} is strict.

\begin{corollary}\label{cor:strict-symbol}
Let $m\in\mathbb R$, $1\leq\lambda\leq2$, $r>0$,
$N>2n/\lambda$, and $0\leq\delta<1$. Then
\[
BS^m_{1,\delta}
\subsetneq BS\mathcal B^m_\delta(r,\infty,\infty;N,\lambda).
\]
\end{corollary}
\begin{proof}
The inclusion follows from Proposition~\ref{prop:hormander}.
Choose a noninteger $\alpha>r$ and a function
$\zeta\in C_c^\infty(\mathbb R^n)$ equal to one near the origin.
Set $b(x)=\zeta(x)\left|x_1\right|^\alpha$, $\nu=\lfloor\alpha\rfloor$, and
$\sigma=\alpha-\nu\in(0,1)$. All derivatives of $b$ of order at most
$\nu$ are bounded, and those of order $\nu$ satisfy
\[
\sup_{\substack{x,y\in\mathbb R^n\\x\ne y}}
\frac{\left|\partial^\gamma b(x)-\partial^\gamma b(y)\right|}{\left|x-y\right|^\sigma}
<\infty,\qquad \gamma\in\mathbb N_0^n,\quad \left|\gamma\right|=\nu.
\]
The H\"older characterization of Besov spaces
(\cite[Section~2.3.5, p.~51]{T83}) and $\alpha>r$ give
\[
b\in B^\alpha_{\infty,\infty}(\mathbb R^n)
\subset B^r_{\infty,\infty}(\mathbb R^n),
\qquad b\in L^\infty(\mathbb R^n).
\]

Example~\ref{ex:rough-coefficient} therefore gives
\[
a(x,\xi,\eta)=b(x)(1+\left|\xi\right|^2+\left|\eta\right|^2)^{m/2}
\in BS\mathcal B^m_0(r,\infty,\infty;N,\lambda).
\]
For the functions $b_k$ defined in \eqref{eq:profile},
$2^{-k(m+\delta r)}\leq2^{-km}$ for every $k\in\mathbb N_0$.
Hence
\[
\sup_{k\geq0}2^{-k(m+\delta r)}
\left\|b_k\right\|_{B^r_{\infty,\infty}(B^N_{\lambda,\infty})}
\leq\sup_{k\geq0}2^{-km}
\left\|b_k\right\|_{B^r_{\infty,\infty}(B^N_{\lambda,\infty})}<\infty.
\]
Since \eqref{eq:symbol-uniform} is independent of $\delta$, this proves
$a\in BS\mathcal B^m_\delta(r,\infty,\infty;N,\lambda)$.

On the other hand, $b$ is not smooth. Indeed, writing
$e_1=(1,0,\ldots,0)\in\mathbb R^n$, for sufficiently small $t>0$ we have
\[
\partial_{x_1}^{\nu+1}b(te_1)
=\left(\prod_{s=0}^{\nu}(\alpha-s)\right)t^{\sigma-1}
\longrightarrow+\infty\qquad\text{as }t\downarrow0.
\]
If $a$ belonged to $BS^m_{1,\delta}$, its smoothness would imply
\[
b=a(\cdot,0,0)\in C^\infty(\mathbb R^n),
\]
which is a contradiction. Moreover, $a$ is continuous in all its
variables. If it agreed almost everywhere with a smooth symbol,
continuity would force equality everywhere, leading to the same
contradiction. Thus
$a\notin BS^m_{1,\delta}$, and the inclusion is strict.
\end{proof}

\section{Proofs of the main results}\label{sec:main-proof}

\subsection{Preliminary estimates}\label{subsec:preliminary}

Our proofs require a series of lemmas. We begin with H\"older's inequalities for mixed spaces with variable
exponents, which follow from \cite[Lemmas~2.16 and~2.17]{XE},
respectively.

\begin{lem}\label{lem:holder}
Let $p,p_1,p_2,q\in\mathcal P_0$, with $p_1^-,p_2^->1$,
and suppose that $1/p(x)=1/p_1(x)+1/p_2(x)$ for almost every
$x\in\mathbb R^n$. For a measurable function $G:\mathbb R^n\to\mathbb C$
and a sequence $(u_k)_{k\in\mathbb N_0}$ of measurable functions
$u_k:\mathbb R^n\to\mathbb C$, both inequalities hold:
\[
\begin{split}
\left\|(u_kG)_{k\in\mathbb N_0}\right\|_{L^{p(\cdot)}(\ell^{q(\cdot)})}
&\lesssim\left\|(u_k)_{k\in\mathbb N_0}\right\|_{L^{p_1(\cdot)}(\ell^{q(\cdot)})}
\left\|G\right\|_{L^{p_2(\cdot)}},\\
\left\|(u_kG)_{k\in\mathbb N_0}\right\|_{\ell^{q(\cdot)}(L^{p(\cdot)})}
&\lesssim\left\|(u_k)_{k\in\mathbb N_0}\right\|_{\ell^{q(\cdot)}(L^{p_1(\cdot)})}
\left\|G\right\|_{L^{p_2(\cdot)}}.
\end{split}
\]
\end{lem}

The next lemma is about embeddings. The Besov embedding is stated in \cite[Section~6, before
Proposition~6.9]{c}. The Triebel--Lizorkin embedding follows from
the Besov embedding and \cite[Theorem~6.1(iii)]{c}.

\begin{lem}\label{lem:positive-embedding}
Let $p,q\in\mathcal P_0^{\log}$ with $p^->1$, and let
$\alpha\in C^{\log}_{\rm loc}\cap L^\infty$ with
$\alpha^-:=\operatorname*{ess\,inf}_{x\in\mathbb R^n}\alpha(x)>0$.
Then
\[
B^{\alpha(\cdot)}_{p(\cdot),q(\cdot)}\hookrightarrow L^{p(\cdot)},
\qquad
F^{\alpha(\cdot)}_{p(\cdot),q(\cdot)}\hookrightarrow L^{p(\cdot)}.
\]
Both embeddings are continuous, with constants depending only on
$p$, $q$, $\alpha$, and the dyadic kernels.
\end{lem}

The following estimates are given in \cite[Lemma~3.4]{AC}.

\begin{lem}\label{lem:discrete}
Let $p,q\in\mathcal P_0$ with $p^->1$, and let $\epsilon>0$.
For a sequence of nonnegative measurable functions
$u_k:\mathbb R^n\to[0,\infty)$, $k\in\mathbb N_0$, put
\[
v_j(x)=\sum_{k=0}^\infty2^{-\epsilon\left|j-k\right|}u_k(x),
\qquad j\in\mathbb N_0,\quad x\in\mathbb R^n.
\]
Then
\[
\begin{split}
\left\|(v_j)_{j\in\mathbb N_0}\right\|_{L^{p(\cdot)}(\ell^{q(\cdot)})}
&\lesssim\left\|(u_k)_{k\in\mathbb N_0}\right\|_{L^{p(\cdot)}(\ell^{q(\cdot)})},\\
\left\|(v_j)_{j\in\mathbb N_0}\right\|_{\ell^{q(\cdot)}(L^{p(\cdot)})}
&\lesssim\left\|(u_k)_{k\in\mathbb N_0}\right\|_{\ell^{q(\cdot)}(L^{p(\cdot)})}.
\end{split}
\]
\end{lem}

For the following convolution and Peetre estimates, use the spatial
kernels $\varphi_j=\varphi_j^{(n)}$, $j\in\mathbb N_0$, fixed in
Section~\ref{sec:definitions}.
For $M>n$ and $j\in\mathbb N_0$, put
\[
\eta_{j,M}(x)=2^{jn}(1+2^j\left|x\right|)^{-M},\qquad x\in\mathbb R^n.
\]

The convolution estimates below are given in
\cite[Lemma~2.2]{AC}, and the Peetre characterizations are given in
\cite[Theorem~3.1]{AC}, applied to the weights $2^{k\alpha(x)}$.

\begin{lem}\label{lem:standard}
Let $p,q\in\mathcal P_0^{\log}$ with $p^-,q^->1$, and let
$\alpha\in C^{\log}_{\rm loc}\cap L^\infty$.
Fix local log-H\"older constants $c_q,c_\alpha\geq0$ for $1/q$ and
$\alpha$, respectively, and choose $c_\alpha$ large enough that
\[
2^{j\alpha(x)}\leq C2^{j\alpha(y)}(1+2^j\left|x-y\right|)^{c_\alpha},
\qquad j\in\mathbb N_0,\quad x,y\in\mathbb R^n,
\]
where $C>0$ is independent of $j$, $x$, and $y$.
The following facts hold.

\textup{(i)} For $M>n+c_q$, convolution
$(u_j)_{j\in\mathbb N_0}\mapsto(\eta_{j,M}*u_j)_{j\in\mathbb N_0}$ is bounded on both
$L^{p(\cdot)}(\ell^{q(\cdot)})$ and
$\ell^{q(\cdot)}(L^{p(\cdot)})$.

\textup{(ii)} For $f\in\mathcal S'(\mathbb R^n)$, $k\in\mathbb N_0$,
and $A>0$, define
\[
f_{k,A}^*(x)=\sup_{y\in\mathbb R^n}
\frac{\left|(\varphi_k*f)(x-y)\right|}{(1+2^k\left|y\right|)^A},\qquad x\in\mathbb R^n.
\]
The quasi-norms obtained by replacing $\left|(\varphi_k*f)\right|$ with $f_{k,A}^*$
are equivalent to the quasi-norms of
$F^{\alpha(\cdot)}_{p(\cdot),q(\cdot)}$ and
$B^{\alpha(\cdot)}_{p(\cdot),q(\cdot)}$, respectively, if
\[
\begin{aligned}
A&>\frac n{\min(p^-,q^-)}+c_\alpha
&&\text{for the Triebel--Lizorkin space},\\
A&>\frac n{p^-}+c_q+c_\alpha
&&\text{for the Besov space}.
\end{aligned}
\]
\end{lem}

The Besov and Triebel--Lizorkin estimates in the following lemma
follow from \cite[Theorem~8.1]{c} and
\cite[Theorem~2.15]{XE}, respectively.

\begin{lem}\label{lem:balls}
Let $p,q\in\mathcal P_0^{\log}$ with $p^-,q^->1$, let
$s\in C^{\log}_{\rm loc}\cap L^\infty$ with $s^->0$, and let $C_0>0$.
Suppose that $(h_k)_{k\in\mathbb N_0}$ is a sequence of locally
integrable tempered distributions satisfying
\[
\operatorname{supp}\widehat h_k\subset
\left\{\omega\in\mathbb R^n:\left|\omega\right|\leq C_0 2^k\right\},\qquad k\in\mathbb N_0.
\]
Then
\begin{equation}\label{eq:ball-synthesis}
\begin{split}
\left\|\sum_{k=0}^{\infty}h_k\right\|_{F^{s(\cdot)}_{p(\cdot),q(\cdot)}}
&\lesssim\left\|(2^{ks(\cdot)}h_k)_{k\in\mathbb N_0}\right\|_{L^{p(\cdot)}(\ell^{q(\cdot)})},\\
\left\|\sum_{k=0}^{\infty}h_k\right\|_{B^{s(\cdot)}_{p(\cdot),q(\cdot)}}
&\lesssim\left\|(2^{ks(\cdot)}h_k)_{k\in\mathbb N_0}\right\|_{\ell^{q(\cdot)}(L^{p(\cdot)})}.
\end{split}
\end{equation}
If the respective right-hand side is finite, the series converges in
$L^{p(\cdot)}(\mathbb R^n)$ and in $\mathcal S'(\mathbb R^n)$.
\end{lem}

For the following weighted joint-frequency kernel estimate, $\varphi_j^{(2n)}$,
$j\in\mathbb N_0$, denotes the dyadic kernel on the joint variables
$(\xi,\eta)\in\mathbb R^{2n}$ fixed in Section~\ref{sec:definitions}.

\begin{lem}\label{lem:kernel}
Let $1\leq\lambda\leq2$, $A>0$, and $N>2n/\lambda+2A$.
If $b\in B^N_{\lambda,\infty}(\mathbb R^{2n})$, then there exists a
constant $C>0$, independent of $b$, such that
\begin{equation}\label{eq:kernel-moment}
\int_{\mathbb R^{2n}}\left|\mathcal F^{-1}_{\xi,\eta}b(y,z)\right|
(1+\left|y\right|)^A(1+\left|z\right|)^A\,\mathrm d y\,\mathrm d z
\leq C\left\|b\right\|_{B^N_{\lambda,\infty}(\mathbb R^{2n})}.
\end{equation}
\end{lem}
\begin{proof}
For every $j\in\mathbb N_0$, we have in $\mathcal S'(\mathbb R^{2n})$
\[
\mathcal F^{-1}_{\xi,\eta}(\varphi_j^{(2n)}*b)(y,z)
=\widehat{\varphi_j^{(2n)}}(-y,-z)
\mathcal F^{-1}_{\xi,\eta}b(y,z).
\]
The function on the left is supported in
$\left\{(y,z)\in\mathbb R^{2n}:\left|(y,z)\right|\leq C2^j\right\}$.
Let $\lambda'\in[2,\infty]$ satisfy
$1/\lambda+1/\lambda'=1$, with $\lambda'=\infty$ when $\lambda=1$.
H\"older's and Hausdorff--Young inequalities give
\[
\begin{aligned}
&\int_{\mathbb R^{2n}}
\left|\mathcal F^{-1}_{\xi,\eta}(\varphi_j^{(2n)}*b)(y,z)\right|
(1+\left|y\right|)^A(1+\left|z\right|)^A\,\mathrm d y\,\mathrm d z
\leq C2^{j(2A+2n/\lambda)}
\left\|\mathcal F^{-1}_{\xi,\eta}(\varphi_j^{(2n)}*b)\right\|_{L^{\lambda'}(\mathbb R^{2n})}\\
&\quad \leq C2^{j(2A+2n/\lambda)}
\left\|\varphi_j^{(2n)}*b\right\|_{L^\lambda(\mathbb R^{2n})}
\leq C2^{-j(N-2A-2n/\lambda)}
\left\|b\right\|_{B^N_{\lambda,\infty}(\mathbb R^{2n})}.
\end{aligned}
\]
Since $N>2A+2n/\lambda$, summing over $j\in\mathbb N_0$ and using
$b=\sum_{j=0}^{\infty}\varphi_j^{(2n)}*b$ in
$\mathcal S'(\mathbb R^{2n})$ proves \eqref{eq:kernel-moment}.
\end{proof}

For $f\in L^1_{\rm loc}(\mathbb R^n)$, the Hardy--Littlewood maximal
operator is
\[
\mathcal Mf(x)=\sup_{t>0}\frac1{\left|B(x,t)\right|}
\int_{B(x,t)}\left|f(y)\right|\,\mathrm d y,\qquad x\in\mathbb R^n.
\]
Here $B(x,t)$ is the Euclidean ball of radius $t>0$ centered at $x$,
and $\left|B(x,t)\right|$ denotes its Lebesgue measure.
For $p\in\mathcal P_0^{\log}$ with $p^->1$, the scalar maximal
inequality in \cite[Theorem~1.2]{DHHMS} gives
\begin{equation}\label{eq:scalar-maximal}
\left\|\mathcal Mg\right\|_{L^{p(\cdot)}}\leq C_p\left\|g\right\|_{L^{p(\cdot)}},
\qquad g\in L^{p(\cdot)}.
\end{equation}
The constant $C_p>0$ depends only on $p$ and $n$.

\begin{lem}\label{lem:lowpass}
Let $(\varphi_j)_{j\in\mathbb N_0}$ be the spatial dyadic kernels
associated with the cutoff $\psi$ in Section~\ref{subsec:variable-spaces}.
If $A>n$ and $g\in L^1_{\rm loc}(\mathbb R^n)$, then at every
$x\in\mathbb R^n$ for which $\mathcal Mg(x)<\infty$, and for every
$k\in\mathbb N_0$ and
$z\in\mathbb R^n$,
\begin{equation}\label{eq:lowpass-shift}
\left|\left(\sum_{j=0}^k\varphi_j*g\right)(x-2^{-k}z)\right|\lesssim(1+\left|z\right|)^A\mathcal Mg(x).
\end{equation}
\end{lem}
\begin{proof}
The convolution kernel of $\sum_{j=0}^k\varphi_j*g$ is
$2^{kn}\check\psi(2^k\cdot)$, where $\check\psi=\mathcal F^{-1}\psi$.
Its absolute value
is bounded by
\[
C\sum_{j=0}^{\infty}2^{-jL}2^{kn}
\mathbf1_{\left\{\left|y\right|\leq2^{j-k}\right\}},
\]
where $L>n$ is a fixed real number, $y\in\mathbb R^n$, and
$j\in\mathbb N_0$. A ball centered at $x-2^{-k}z$ of radius $2^{j-k}$
is contained in the ball centered at $x$ of radius
$2^{-k}(2^j+\left|z\right|)$. It follows that
\[
\left|\left(\sum_{j=0}^k\varphi_j*g\right)(x-2^{-k}z)\right|
\lesssim\sum_{j=0}^{\infty}2^{-jL}(2^j+\left|z\right|)^n\mathcal Mg(x)
\lesssim(1+\left|z\right|)^n\mathcal Mg(x).
\]
This implies \eqref{eq:lowpass-shift}.
\end{proof}

\Needspace{4\baselineskip}
The following lemma gives estimates for localized symbols.

\begin{lem}\label{lem:localized}
Let $m\in\mathbb R$, $0\leq\delta<1$, $1\leq\mu\leq\infty$,
$0<v\leq\infty$, $1\leq\lambda\leq2$, $r>n/\mu$, and $N>2n/\lambda$.
Let $a\in BS\mathcal B^m_\delta(r,\mu,v;N,\lambda)$.
Let $k,\ell\in\mathbb N_0$.
Let $\psi$ be the cutoff used to define the dyadic resolution in Section~\ref{sec:definitions}.
Let $(\varphi_j)_{j\in\mathbb N_0}$ be the corresponding spatial dyadic kernels.
Choose
$\theta\in C_c^\infty(\mathbb R^n)$ equal to one on
$\{1/2\leq\left|\xi\right|\leq2\}$ and supported in
$\{1/4\leq\left|\xi\right|\leq8\}$. Put $\theta_k(\xi)=\theta(2^{-k}\xi)$
for $k\geq1$, and $\theta_0(\xi)=\psi(\xi/2)$. Define
\[
c_k(x,\xi,\eta)=a(x,\xi,\eta)\theta_k(\xi)\psi(2^{-k-1}\eta),
\quad c_{\ell,k}=\varphi_\ell*_x c_k,\quad
k,\ell\in\mathbb N_0.
\]
For all $f,g\in\mathcal S(\mathbb R^n)$,
\[
T_{c_k}\left(\varphi_k*f,\sum_{j=0}^k\varphi_j*g\right)
=T_a\left(\varphi_k*f,\sum_{j=0}^k\varphi_j*g\right).
\]
Moreover,
\begin{equation}\label{eq:localized-low}
\mathop{\rm ess\,sup}_{x\in\mathbb R^n}
\left\|\sum_{\nu=0}^k c_{\nu,k}(x,2^k\cdot,2^k\cdot)\right\|_{B^N_{\lambda,\infty}}
\lesssim\left\|a\right\|_{BS\mathcal B^m_\delta(r,\mu,v;N,\lambda)}2^{km}.
\end{equation}
One also has
\begin{equation}\label{eq:localized-high}
\mathop{\rm ess\,sup}_{x\in\mathbb R^n}
\left\|c_{\ell,k}(x,2^k\cdot,2^k\cdot)\right\|_{B^N_{\lambda,\infty}}
\lesssim\left\|a\right\|_{BS\mathcal B^m_\delta(r,\mu,v;N,\lambda)}2^{-\ell(r-n/\mu)}2^{k(m+\delta r)}.
\end{equation}
The implicit constants in these two estimates depend only on the
parameters and the fixed cutoffs, and not on $a$, $k$, or $\ell$.
\end{lem}
\begin{proof}
Since $\sum_{j=0}^k\widehat\varphi_j=\psi(2^{-k}\cdot)$, for
$k\in\mathbb N_0$ and $\xi,\eta\in\mathbb R^n$ we have
\[
\theta_k(\xi)\widehat\varphi_k(\xi)=\widehat\varphi_k(\xi),\qquad
\psi(2^{-k-1}\eta)\psi(2^{-k}\eta)=\psi(2^{-k}\eta).
\]
This proves the identity between the two localized operators.
The supports of $\theta_k$ and $\psi$ give
\[
\operatorname{supp}c_k(x,\cdot,\cdot)
\subset\left\{(\xi,\eta)\in\mathbb R^{2n}:2^{k-2}\leq\left|(\xi,\eta)\right|\leq2^{k+2}\sqrt5\right\},
\quad k\geq1,
\]
and $\operatorname{supp}c_0(x,\cdot,\cdot)\subset
\left\{(\xi,\eta)\in\mathbb R^{2n}:\left|(\xi,\eta)\right|\leq4\sqrt2\right\}$.
Consequently,
\[
\chi_h(\xi,\eta)c_k(x,\xi,\eta)=0,
\qquad h,k\in\mathbb N_0,\quad \left|h-k\right|>4.
\]
Using $b_h$ from \eqref{eq:profile}, we obtain
\[
c_k(x,2^k\xi,2^k\eta)
=\theta_k(2^k\xi)\psi(\eta/2)
\sum_{\substack{h\in\mathbb N_0\\\left|h-k\right|\leq4}}
b_h(x,2^{k-h}\xi,2^{k-h}\eta),\qquad k\in\mathbb N_0.
\]
The normalized cutoffs have uniform derivative bounds, and only the
finitely many dilations $2^{k-h}$ with $\left|h-k\right|\leq4$ occur.
Smooth multiplication and dilation estimates therefore give
\[
\left\|c_k(x,2^k\cdot,2^k\cdot)\right\|_{B^N_{\lambda,\infty}}
\lesssim\sum_{\substack{h\in\mathbb N_0\\\left|h-k\right|\leq4}}
\left\|b_h(x,\cdot,\cdot)\right\|_{B^N_{\lambda,\infty}}.
\]
These operations commute with spatial convolution. Hence
\[
\left\|c_k(\cdot,2^k\cdot,2^k\cdot)\right\|_{B^r_{\mu,v}(B^N_{\lambda,\infty})}
\lesssim\sum_{\substack{h\in\mathbb N_0\\\left|h-k\right|\leq4}}
\left\|b_h\right\|_{B^r_{\mu,v}(B^N_{\lambda,\infty})}
\lesssim2^{k(m+\delta r)}
\left\|a\right\|_{BS\mathcal B^m_\delta(r,\mu,v;N,\lambda)}.
\]
For the spatial low-frequency part, the dyadic resolution gives
\[
\sum_{\nu=0}^k\varphi_\nu(x)
=2^{kn}(\mathcal F^{-1}\psi)(2^kx),\qquad
\left\|\sum_{\nu=0}^k\varphi_\nu\right\|_{L^1}
=\left\|\mathcal F^{-1}\psi\right\|_{L^1}.
\]
Minkowski's inequality and \eqref{eq:symbol-uniform} yield
\[
\begin{aligned}
\mathop{\rm ess\,sup}_{x\in\mathbb R^n}
\left\|\sum_{\nu=0}^k c_{\nu,k}(x,2^k\cdot,2^k\cdot)\right\|_{B^N_{\lambda,\infty}}
&\lesssim\left\|\mathcal F^{-1}\psi\right\|_{L^1}
\sum_{\substack{h\in\mathbb N_0\\\left|h-k\right|\leq4}}
\mathop{\rm ess\,sup}_{x\in\mathbb R^n}
\left\|b_h(x,\cdot,\cdot)\right\|_{B^N_{\lambda,\infty}}\\
&\lesssim2^{km}\left\|a\right\|_{BS\mathcal B^m_\delta(r,\mu,v;N,\lambda)}.
\end{aligned}
\]
This proves \eqref{eq:localized-low}.
For the second bound, note that
\[
\left\|\left\|c_{\ell,k}(x,2^k\cdot,2^k\cdot)\right\|_{B^N_{\lambda,\infty}}
\right\|_{L^\mu_x}
\lesssim\left\|a\right\|_{BS\mathcal B^m_\delta(r,\mu,v;N,\lambda)}2^{-\ell r}2^{k(m+\delta r)}.
\]
To verify the Bernstein estimate, choose $\sigma\in\mathcal S(\mathbb R^n)$
with $\widehat\sigma=1$ on $\left\{\zeta\in\mathbb R^n:\left|\zeta\right|\leq2\right\}$, and put
$\sigma_\ell(x)=2^{\ell n}\sigma(2^\ell x)$.
The spatial Fourier support of $c_{\ell,k}$ gives
$c_{\ell,k}=\sigma_\ell*_x c_{\ell,k}$.
H\"older's inequality yields
\[
\begin{aligned}
\left\|c_{\ell,k}(x,2^k\cdot,2^k\cdot)\right\|_{B^N_{\lambda,\infty}}
&\leq\int_{\mathbb R^n}\left|\sigma_\ell(x-y)\right|
\left\|c_{\ell,k}(y,2^k\cdot,2^k\cdot)\right\|_{B^N_{\lambda,\infty}}\,\mathrm d y\\
&\leq\left\|\sigma_\ell\right\|_{L^{\mu'}}
\left\|\left\|c_{\ell,k}(\cdot,2^k\cdot,2^k\cdot)\right\|_{B^N_{\lambda,\infty}}
\right\|_{L^\mu}.
\end{aligned}
\]
Since $\left\|\sigma_\ell\right\|_{L^{\mu'}}=2^{\ell n/\mu}\left\|\sigma\right\|_{L^{\mu'}}$,
we get
\[
\mathop{\rm ess\,sup}_{x\in\mathbb R^n}\left\|c_{\ell,k}(x,2^k\cdot,2^k\cdot)\right\|_{B^N_{\lambda,\infty}}
\lesssim 2^{\ell n/\mu}
\left\|\left\|c_{\ell,k}(x,2^k\cdot,2^k\cdot)\right\|_{B^N_{\lambda,\infty}}
\right\|_{L^\mu_x}.
\]
This proves \eqref{eq:localized-high} and the proof is complete.
\end{proof}

\subsection{Proofs}\label{subsec:proofs}

\begin{proof}[Proof of Theorem~\ref{thm:triebel}]
We first suppose that $a\in BS\mathcal B^m_\delta(r,\mu,v;N,\lambda)$,
with the parameters as in the theorem. Let $f,g\in\mathcal S(\mathbb R^n)$.
By \eqref{eq:frequency-threshold}, choose $A$ such that
\begin{equation}\label{eq:peetre-choice}
n+c_s+c_q<A<\frac12\left(N-\frac{2n}{\lambda}\right).
\end{equation}
Since $p_1^-,p_2^-,q^->1$, we also have
\[
A>\max_{i\in\{1,2\}}\left\{
\frac n{\min\left\{p_i^-,q^-\right\}}+c_s,\frac n{p_i^-}+c_s+c_q\right\}.
\]
The dyadic cutoffs and \eqref{eq:operator} satisfy, for almost every $x$,
\[
\sum_{j,k=0}^{\infty}\left|\widehat\varphi_k(\xi)\widehat\varphi_j(\eta)\right|=1,
\qquad
\int_{\mathbb R^{2n}}\left|a(x,\xi,\eta)\widehat f(\xi)\widehat g(\eta)\right|\,\mathrm d\xi\,\mathrm d\eta<\infty.
\]
Splitting the input indices into $j\leq k$ and $k<j$ therefore gives
\begin{equation}\label{eq:two-input-split}
T_a(f,g)=\sum_{k=0}^{\infty}T_a\left(\varphi_k*f,\sum_{j=0}^k\varphi_j*g\right)
+\sum_{j=0}^{\infty}T_a\left(\sum_{i=0}^{j-1}\varphi_i*f,\varphi_j*g\right).
\end{equation}
The symbol class and its quasi-norm are invariant under
$\xi\leftrightarrow\eta$. Hence it suffices to estimate the first sum
in \eqref{eq:two-input-split}; the second follows by interchanging
$(f,p_1)$ and $(g,p_2)$ and applying Lemma~\ref{lem:lowpass} at
scale $j-1$, $j\geq1$.

Using $c_k,c_{\ell,k}$ from Lemma~\ref{lem:localized}, define
\[
h_k=T_{\sum_{\nu=0}^k c_{\nu,k}}\left(\varphi_k*f,\sum_{j=0}^k\varphi_j*g\right),\qquad
H_\ell=\sum_{k=0}^{\ell-1}T_{c_{\ell,k}}\left(\varphi_k*f,\sum_{j=0}^k\varphi_j*g\right),
\]
where $k\in\mathbb N_0$, $\ell\in\mathbb N$, and $H_0=0$.
The spatial decomposition gives
\begin{equation}\label{eq:space-split}
\sum_{k=0}^{\infty}T_a\left(\varphi_k*f,\sum_{j=0}^k\varphi_j*g\right)
=\sum_{k=0}^{\infty}h_k+\sum_{\ell=1}^{\infty}H_\ell.
\end{equation}
We first estimate finite partial sums; convergence and the identity
\eqref{eq:space-split} are justified below.

For $k\in\mathbb N_0$, let $d_k=\sum_{\nu=0}^k c_{\nu,k}$ or
$d_k=c_{\ell,k}$ with $\ell\in\mathbb N$, $\ell>k$.
Changing variables in the kernel representation gives, for $x\in\mathbb R^n$,
\[
\begin{aligned}
&T_{d_k}\left(\varphi_k*f,\sum_{j=0}^k\varphi_j*g\right)(x)\\
&\quad=\int_{\mathbb R^{2n}}\mathcal F^{-1}_{\xi,\eta}[d_k(x,2^k\xi,2^k\eta)](y,z)
(\varphi_k*f)(x-2^{-k}y)\left(\sum_{j=0}^k\varphi_j*g\right)(x-2^{-k}z)\,\mathrm d y\,\mathrm d z.
\end{aligned}
\]
The definition of $f_{k,A}^*$ and Lemma~\ref{lem:lowpass} yield
\[
\left|(\varphi_k*f)(x-2^{-k}y)\right|\leq(1+\left|y\right|)^Af_{k,A}^*(x),\qquad
\left|\left(\sum_{j=0}^k\varphi_j*g\right)(x-2^{-k}z)\right|
\lesssim(1+\left|z\right|)^A\mathcal Mg(x).
\]
By \eqref{eq:peetre-choice}, $N>2n/\lambda+2A$, so Lemma~\ref{lem:kernel} gives
\begin{equation}\label{eq:pointwise-localized}
\begin{aligned}
&\left|T_{d_k}\left(\varphi_k*f,\sum_{j=0}^k\varphi_j*g\right)(x)\right|\\
&\quad\lesssim f_{k,A}^*(x)\mathcal Mg(x)
\int_{\mathbb R^{2n}}\left|\mathcal F^{-1}_{\xi,\eta}[d_k(x,2^k\xi,2^k\eta)](y,z)\right|(1+\left|y\right|)^A(1+\left|z\right|)^A\,\mathrm d y\,\mathrm d z\\
&\quad\lesssim\left\|d_k(x,2^k\cdot,2^k\cdot)\right\|_{B^N_{\lambda,\infty}}
f_{k,A}^*(x)\mathcal Mg(x).
\end{aligned}
\end{equation}

Let $\zeta\in\mathbb R^n$ be the Fourier variable dual to $x$.
The spatial cutoffs give, in $\mathcal S'$, 
\[
\begin{aligned}
\mathcal F_x\left(\sum_{\nu=0}^k c_{\nu,k}\right)(\zeta,\xi,\eta)
&=\psi(2^{-k}\zeta)(\mathcal F_xc_k)(\zeta,\xi,\eta),\\
(\mathcal F_xc_{\ell,k})(\zeta,\xi,\eta)
&=\widehat\varphi_\ell(\zeta)(\mathcal F_xc_k)(\zeta,\xi,\eta).
\end{aligned}
\]
Moreover,
\[
\mathcal F_x[e^{2\pi ix\cdot(\xi+\eta)}d_k](\omega,\xi,\eta)
=(\mathcal F_xd_k)(\omega-\xi-\eta,\xi,\eta),\qquad\omega\in\mathbb R^n.
\]
On the input supports, $\left|\xi\right|,\left|\eta\right|\leq2^{k+1}$.
Hence $\omega=\zeta+\xi+\eta$ satisfies
\[
\begin{aligned}
\left|\omega\right|&\leq2^{k+1}+2^{k+2}=6\cdot2^k
&&\text{for }h_k,\\
\left|\omega\right|&\leq2^{\ell+1}+2^{k+2}\leq4\cdot2^\ell
&&\text{for each summand of }H_\ell\ (k<\ell).
\end{aligned}
\]
Thus, with $C_0=6$,
\begin{equation}\label{eq:output-balls}
\operatorname{supp}\widehat h_k\subset\left\{\omega\in\mathbb R^n:\left|\omega\right|\leq C_0 2^k\right\},\qquad
\operatorname{supp}\widehat H_\ell\subset\left\{\omega\in\mathbb R^n:\left|\omega\right|\leq C_0 2^\ell\right\}.
\end{equation}

By \eqref{eq:localized-low} and \eqref{eq:pointwise-localized},
\[
2^{ks(x)}\left|h_k(x)\right|
\lesssim\left\|a\right\|_{BS\mathcal B^m_\delta(r,\mu,v;N,\lambda)}
2^{k(s(x)+m)}f_{k,A}^*(x)\mathcal Mg(x).
\]
Using \eqref{eq:output-balls}, Lemmas~\ref{lem:balls},
\ref{lem:holder}, \ref{lem:standard}(ii), and
\eqref{eq:scalar-maximal}, we obtain
\[
\begin{aligned}
&\left\|\sum_{k=0}^{\infty}h_k\right\|_{F^{s(\cdot)}_{p(\cdot),q(\cdot)}}
\lesssim\left\|(2^{ks(\cdot)}h_k)_{k\in\mathbb N_0}\right\|_{L^{p(\cdot)}(\ell^{q(\cdot)})}\\
&\quad \lesssim\left\|a\right\|_{BS\mathcal B^m_\delta(r,\mu,v;N,\lambda)}
\left\|(2^{k(s(\cdot)+m)}f_{k,A}^*)_{k\in\mathbb N_0}\right\|_{L^{p_1(\cdot)}(\ell^{q(\cdot)})}
\left\|\mathcal Mg\right\|_{L^{p_2(\cdot)}}
\lesssim\left\|a\right\|_{BS\mathcal B^m_\delta(r,\mu,v;N,\lambda)}
\left\|f\right\|_{F^{s(\cdot)+m}_{p_1(\cdot),q(\cdot)}}\left\|g\right\|_{L^{p_2(\cdot)}}.
\end{aligned}
\]

By \eqref{eq:smoothness-range},
\[
r-\frac n\mu>s^+,\qquad (1-\delta)r-\frac n\mu\geq0.
\]
For $0\leq k<\ell$ and $x\in\mathbb R^n$,
\[
\begin{aligned}
\ell s(x)-\ell(r-n/\mu)+k(m+\delta r)
&=-(\ell-k)(r-n/\mu-s(x))-k((1-\delta)r-n/\mu)+k(s(x)+m)\\
&\leq-(\ell-k)(r-n/\mu-s^+)+k(s(x)+m).
\end{aligned}
\]
Consequently, \eqref{eq:localized-high} and
\eqref{eq:pointwise-localized} imply
\[
2^{\ell s(x)}\left|H_\ell(x)\right|
\lesssim\left\|a\right\|_{BS\mathcal B^m_\delta(r,\mu,v;N,\lambda)}\mathcal Mg(x)
\sum_{k=0}^{\ell-1}2^{-(\ell-k)(r-n/\mu-s^+)}2^{k(s(x)+m)}f_{k,A}^*(x).
\]
Lemma~\ref{lem:discrete}, with $\epsilon=r-n/\mu-s^+>0$, gives
\[
\left\|\left(\sum_{k=0}^{\ell-1}2^{-(\ell-k)(r-n/\mu-s^+)}
2^{k(s(\cdot)+m)}f_{k,A}^*\right)_{\ell\in\mathbb N_0}\right\|_{L^{p_1(\cdot)}(\ell^{q(\cdot)})}
\lesssim\left\|(2^{k(s(\cdot)+m)}f_{k,A}^*)_{k\in\mathbb N_0}\right\|_{L^{p_1(\cdot)}(\ell^{q(\cdot)})}.
\]
Together with \eqref{eq:output-balls}, Lemmas~\ref{lem:balls},
\ref{lem:holder}, \ref{lem:standard}(ii), and
\eqref{eq:scalar-maximal}, this yields
\[
\begin{aligned}
&\left\|\sum_{\ell=1}^{\infty}H_\ell\right\|_{F^{s(\cdot)}_{p(\cdot),q(\cdot)}}
\lesssim\left\|(2^{\ell s(\cdot)}H_\ell)_{\ell\in\mathbb N_0}\right\|_{L^{p(\cdot)}(\ell^{q(\cdot)})}\\
&\quad \lesssim\left\|a\right\|_{BS\mathcal B^m_\delta(r,\mu,v;N,\lambda)}
\left\|(2^{k(s(\cdot)+m)}f_{k,A}^*)_{k\in\mathbb N_0}\right\|_{L^{p_1(\cdot)}(\ell^{q(\cdot)})}
\left\|\mathcal Mg\right\|_{L^{p_2(\cdot)}}
\lesssim\left\|a\right\|_{BS\mathcal B^m_\delta(r,\mu,v;N,\lambda)}
\left\|f\right\|_{F^{s(\cdot)+m}_{p_1(\cdot),q(\cdot)}}\left\|g\right\|_{L^{p_2(\cdot)}}.
\end{aligned}
\]
Lemma~\ref{lem:balls} gives convergence of both spatial series in
$L^{p(\cdot)}$ and $\mathcal S'$.
For fixed $k\in\mathbb N_0$, \eqref{eq:localized-high} also gives
\[
\sum_{\ell=L+1}^{\infty}\mathop{\rm ess\,sup}_{x\in\mathbb R^n}
\left\|c_{\ell,k}(x,2^k\cdot,2^k\cdot)\right\|_{B^N_{\lambda,\infty}}
\lesssim\left\|a\right\|_{BS\mathcal B^m_\delta(r,\mu,v;N,\lambda)}2^{k(m+\delta r)}2^{-L(r-n/\mu)}
\longrightarrow0\quad(L\to\infty).
\]
Since $c_k=\sum_{\ell=0}^{\infty}c_{\ell,k}$ in $\mathcal S'$,
this estimate and \eqref{eq:pointwise-localized} imply
\[
T_{c_k}\left(\varphi_k*f,\sum_{j=0}^k\varphi_j*g\right)
=\sum_{\ell=0}^{\infty}T_{c_{\ell,k}}\left(\varphi_k*f,\sum_{j=0}^k\varphi_j*g\right)
\quad\text{almost everywhere}.
\]
For every $M>0$ and $k\geq1$, the Schwartz decay gives
\[
\sup_{x\in\mathbb R^n}f_{k,A}^*(x)
\leq\left\|\widehat\varphi_k\widehat f\right\|_{L^1}
\lesssim_M2^{-kM}.
\]
Using \eqref{eq:localized-low}, \eqref{eq:localized-high}, and
$(1-\delta)r-n/\mu\geq0$, we obtain, for almost every $x$,
\[
\sum_{k=0}^{\infty}\left(\left|h_k(x)\right|+
\sum_{\ell=k+1}^{\infty}\left|T_{c_{\ell,k}}\left(\varphi_k*f,\sum_{j=0}^k\varphi_j*g\right)(x)\right|\right)
\lesssim\left\|a\right\|_{BS\mathcal B^m_\delta(r,\mu,v;N,\lambda)}\mathcal Mg(x)
\sum_{k=0}^{\infty}2^{km}f_{k,A}^*(x)<\infty.
\]
This justifies the regrouping in \eqref{eq:space-split}.

By \eqref{eq:space-split} and symmetry, \eqref{eq:two-input-split}
yields \eqref{eq:triebel-main}.

For the Besov estimate, the second inequalities in
Lemmas~\ref{lem:holder} and \ref{lem:discrete}, together with
Lemma~\ref{lem:standard}(ii) and \eqref{eq:scalar-maximal}, give
\[
\begin{aligned}
&\left\|(2^{ks(\cdot)}h_k)_{k\in\mathbb N_0}\right\|_{\ell^{q(\cdot)}(L^{p(\cdot)})}
+\left\|(2^{ks(\cdot)}H_k)_{k\in\mathbb N_0}\right\|_{\ell^{q(\cdot)}(L^{p(\cdot)})}\\
&\quad \lesssim\left\|a\right\|_{BS\mathcal B^m_\delta(r,\mu,v;N,\lambda)}
\left\|(2^{k(s(\cdot)+m)}f_{k,A}^*)_{k\in\mathbb N_0}\right\|_{\ell^{q(\cdot)}(L^{p_1(\cdot)})}
\left\|\mathcal Mg\right\|_{L^{p_2(\cdot)}}
\lesssim\left\|a\right\|_{BS\mathcal B^m_\delta(r,\mu,v;N,\lambda)}
\left\|f\right\|_{B^{s(\cdot)+m}_{p_1(\cdot),q(\cdot)}}\left\|g\right\|_{L^{p_2(\cdot)}}.
\end{aligned}
\]
Now \eqref{eq:ball-synthesis} and symmetry prove \eqref{eq:besov-main}
for the Besov symbol class.

Finally, let $1\leq\mu<\infty$ and
$a\in BS\mathcal F^m_\delta(r,\mu,v;N,\lambda)$.
Proposition~\ref{prop:symbol-embeddings}\textup{(iii)} gives
\[
\left\|a\right\|_{BS\mathcal B^m_\delta(r,\mu,\infty;N,\lambda)}
\leq\left\|a\right\|_{BS\mathcal B^m_\delta(r,\mu,\max\{\mu,v\};N,\lambda)}
\lesssim\left\|a\right\|_{BS\mathcal F^m_\delta(r,\mu,v;N,\lambda)}.
\]
Applying the estimates just proved for the Besov symbol class with
$v=\infty$, we obtain
\[
\begin{aligned}
\left\|T_a(f,g)\right\|_{F^{s(\cdot)}_{p(\cdot),q(\cdot)}}
&\lesssim\left\|a\right\|_{BS\mathcal B^m_\delta(r,\mu,\infty;N,\lambda)}
\Bigl(\left\|f\right\|_{F^{s(\cdot)+m}_{p_1(\cdot),q(\cdot)}}\left\|g\right\|_{L^{p_2(\cdot)}}
+\left\|f\right\|_{L^{p_1(\cdot)}}\left\|g\right\|_{F^{s(\cdot)+m}_{p_2(\cdot),q(\cdot)}}\Bigr)\\
&\lesssim\left\|a\right\|_{BS\mathcal F^m_\delta(r,\mu,v;N,\lambda)}
\Bigl(\left\|f\right\|_{F^{s(\cdot)+m}_{p_1(\cdot),q(\cdot)}}\left\|g\right\|_{L^{p_2(\cdot)}}
+\left\|f\right\|_{L^{p_1(\cdot)}}\left\|g\right\|_{F^{s(\cdot)+m}_{p_2(\cdot),q(\cdot)}}\Bigr),
\end{aligned}
\]
and
\[
\begin{aligned}
\left\|T_a(f,g)\right\|_{B^{s(\cdot)}_{p(\cdot),q(\cdot)}}
&\lesssim\left\|a\right\|_{BS\mathcal B^m_\delta(r,\mu,\infty;N,\lambda)}
\Bigl(\left\|f\right\|_{B^{s(\cdot)+m}_{p_1(\cdot),q(\cdot)}}\left\|g\right\|_{L^{p_2(\cdot)}}
+\left\|f\right\|_{L^{p_1(\cdot)}}\left\|g\right\|_{B^{s(\cdot)+m}_{p_2(\cdot),q(\cdot)}}\Bigr)\\
&\lesssim\left\|a\right\|_{BS\mathcal F^m_\delta(r,\mu,v;N,\lambda)}
\Bigl(\left\|f\right\|_{B^{s(\cdot)+m}_{p_1(\cdot),q(\cdot)}}\left\|g\right\|_{L^{p_2(\cdot)}}
+\left\|f\right\|_{L^{p_1(\cdot)}}\left\|g\right\|_{B^{s(\cdot)+m}_{p_2(\cdot),q(\cdot)}}\Bigr).
\end{aligned}
\]
This proves both estimates for the Triebel--Lizorkin symbol class.
\end{proof}

\begin{proof}[Proof of Corollary~\ref{cor:extension}]
Since $s^-+m>0$, Lemma~\ref{lem:positive-embedding} gives, for
$i\in\{1,2\}$,
\[
F^{s(\cdot)+m}_{p_i(\cdot),q(\cdot)}\hookrightarrow L^{p_i(\cdot)},\qquad
B^{s(\cdot)+m}_{p_i(\cdot),q(\cdot)}\hookrightarrow L^{p_i(\cdot)}.
\]
Let $f,g$ belong to the two Triebel--Lizorkin input spaces.
By \cite[Corollary~5.2]{ACatoms}, choose $f_j,g_j\in\mathcal S(\mathbb R^n)$,
$j\in\mathbb N$, such that
\[
f_j\longrightarrow f\quad\text{in }F^{s(\cdot)+m}_{p_1(\cdot),q(\cdot)},\qquad
g_j\longrightarrow g\quad\text{in }F^{s(\cdot)+m}_{p_2(\cdot),q(\cdot)}.
\]
For $j,\ell\in\mathbb N$, bilinearity gives
\[
T_a(f_j,g_j)-T_a(f_\ell,g_\ell)
=T_a(f_j-f_\ell,g_j)+T_a(f_\ell,g_j-g_\ell).
\]
Hence \eqref{eq:triebel-main} and the embeddings yield
\[
\begin{aligned}
&\left\|T_a(f_j,g_j)-T_a(f_\ell,g_\ell)\right\|_{F^{s(\cdot)}_{p(\cdot),q(\cdot)}}\\
&\quad\lesssim\left\|a\right\|_{BS\mathcal A^m_\delta(r,\mu,v;N,\lambda)}
\left(\left\|f_j-f_\ell\right\|_{F^{s(\cdot)+m}_{p_1(\cdot),q(\cdot)}}
\left\|g_j\right\|_{F^{s(\cdot)+m}_{p_2(\cdot),q(\cdot)}}
+\left\|f_\ell\right\|_{F^{s(\cdot)+m}_{p_1(\cdot),q(\cdot)}}
\left\|g_j-g_\ell\right\|_{F^{s(\cdot)+m}_{p_2(\cdot),q(\cdot)}}\right)
\xrightarrow[j,\ell\to\infty]{}0.
\end{aligned}
\]
The target space is complete (see \cite{c,DHR}), so we define
\[
T_a(f,g):=\lim_{j\to\infty}T_a(f_j,g_j)
\quad\text{in }F^{s(\cdot)}_{p(\cdot),q(\cdot)}.
\]
The same difference estimate for any two approximating pairs makes
this limit independent of the choices, and density gives uniqueness.
The embeddings also imply $f_j\to f$ in $L^{p_1(\cdot)}$ and
$g_j\to g$ in $L^{p_2(\cdot)}$; passing to the limit preserves
\eqref{eq:triebel-main}.
Replacing the Triebel--Lizorkin norms by the Besov norms in this
argument proves the Besov extension and \eqref{eq:besov-main}.
\end{proof}

\begin{proof}[Proof of Theorem~\ref{thm:algebra}]
By \cite[Corollary~6.7]{c}, the assumed uniform gap gives
\[
B^{s(\cdot)}_{p(\cdot),q(\cdot)}\hookrightarrow L^\infty.
\]
Moreover, \cite[Theorem~6.1(iii) and Corollary~6.7]{c} yield
\[
F^{s(\cdot)}_{p(\cdot),q(\cdot)}
\hookrightarrow B^{s(\cdot)}_{p(\cdot),\max\{p(\cdot),q(\cdot)\}}
\hookrightarrow L^\infty.
\]

Let $f,g\in\mathcal S(\mathbb R^n)$. Taking $a\equiv1$ in the
decomposition \eqref{eq:two-input-split}, we have
\[
fg=\sum_{k=0}^\infty(\varphi_k*f)\left(\sum_{j=0}^k\varphi_j*g\right)
+\sum_{k=1}^\infty\left(\sum_{j=0}^{k-1}\varphi_j*f\right)(\varphi_k*g).
\]
Each summand at index $k$ has Fourier support in
$\left\{\omega\in\mathbb R^n:\left|\omega\right|\leq2^{k+2}\right\}$.
The low-frequency kernels satisfy
\[
\left\|\sum_{j=0}^k\varphi_j\right\|_{L^1}
=\left\|\mathcal F^{-1}\psi\right\|_{L^1},\qquad
\left\|\sum_{j=0}^k\varphi_j*g\right\|_{L^\infty(\mathbb R^n)}
\lesssim\left\|g\right\|_{L^\infty(\mathbb R^n)},\qquad k\in\mathbb N_0.
\]
Since $s^->0$, Lemma~\ref{lem:balls} and the definition of the mixed
norm give
\[
\left\|\sum_{k=0}^\infty(\varphi_k*f)
\left(\sum_{j=0}^k\varphi_j*g\right)\right\|_{F^{s(\cdot)}_{p(\cdot),q(\cdot)}}
\lesssim\left\|g\right\|_{L^\infty(\mathbb R^n)}
\left\|\left(2^{ks(\cdot)}(\varphi_k*f)\right)_{k\in\mathbb N_0}\right\|_{L^{p(\cdot)}(\ell^{q(\cdot)})}
=\left\|g\right\|_{L^\infty(\mathbb R^n)}
\left\|f\right\|_{F^{s(\cdot)}_{p(\cdot),q(\cdot)}}.
\]
Interchanging $f$ and $g$ and using the low-frequency cutoff at $k-1$
gives the estimate for the second sum. Hence
\begin{equation}\label{eq:product-triebel}
\left\|fg\right\|_{F^{s(\cdot)}_{p(\cdot),q(\cdot)}}
\lesssim\left\|f\right\|_{F^{s(\cdot)}_{p(\cdot),q(\cdot)}}\left\|g\right\|_{L^\infty(\mathbb R^n)}
+\left\|f\right\|_{L^\infty(\mathbb R^n)}\left\|g\right\|_{F^{s(\cdot)}_{p(\cdot),q(\cdot)}}.
\end{equation}
Using the Besov estimate in Lemma~\ref{lem:balls} and the mixed norm
$\ell^{q(\cdot)}(L^{p(\cdot)})$ instead, we obtain
\begin{equation}\label{eq:product-besov}
\left\|fg\right\|_{B^{s(\cdot)}_{p(\cdot),q(\cdot)}}
\lesssim\left\|f\right\|_{B^{s(\cdot)}_{p(\cdot),q(\cdot)}}\left\|g\right\|_{L^\infty(\mathbb R^n)}
+\left\|f\right\|_{L^\infty(\mathbb R^n)}\left\|g\right\|_{B^{s(\cdot)}_{p(\cdot),q(\cdot)}}.
\end{equation}
Combining \eqref{eq:product-triebel} and \eqref{eq:product-besov}
with the $L^\infty$ embeddings proves \eqref{eq:algebra-triebel}
and \eqref{eq:algebra-besov} for Schwartz functions.

By \cite[Corollary~5.2]{ACatoms}, $\mathcal S(\mathbb R^n)$ is dense
in both spaces. The estimates therefore extend multiplication to a
continuous bilinear map on each space. To identify the extension, let
$f_\nu,g_\nu\in\mathcal S(\mathbb R^n)$, $\nu\in\mathbb N$, converge
to $f,g$ in the respective space. The $L^\infty$ embeddings give
\[
\left\|f_\nu g_\nu-fg\right\|_{L^\infty(\mathbb R^n)}
\leq\left\|f_\nu-f\right\|_{L^\infty(\mathbb R^n)}\left\|g_\nu\right\|_{L^\infty(\mathbb R^n)}
+\left\|f\right\|_{L^\infty(\mathbb R^n)}\left\|g_\nu-g\right\|_{L^\infty(\mathbb R^n)}
\longrightarrow0.
\]
The extension also converges in $L^\infty$, so it agrees with the
pointwise product. This proves
\eqref{eq:algebra-triebel} and \eqref{eq:algebra-besov}.
\end{proof}

\section*{Acknowledgements}
Under the authors' guidance, AI tools were used to support discussions
and revisions during manuscript preparation. The authors checked the
final arguments and proofs and take full responsibility for the manuscript.

\noindent The work is supported by the Science and Technology Project of Guangxi
(Guike AD25069086) and the Natural Science Foundation of Guangxi
(No.~2026GXNSFBA00640069).

\section*{Declarations}

\noindent\textbf{Author Contributions}\quad
All authors developed and discussed the results and contributed to the final
manuscript.

\medskip

\noindent\textbf{Data Availability}\quad Data sharing is not applicable
to this article as no data sets were generated or analyzed.

\medskip

\noindent\textbf{Conflict of interest}\quad All authors state no conflict of interest.

\medskip

\noindent\textbf{Informed Consent}\quad Not applicable.


\begin{thebibliography}{99}

\bibitem{AC}
A. Almeida and A. Caetano,
\textit{On 2-microlocal spaces with all exponents variable},
Nonlinear Anal. \textbf{135} (2016), 97--119.

\bibitem{ACatoms}
A. Almeida and A. Caetano,
\textit{Atomic and molecular decompositions in variable exponent
2-microlocal spaces and applications},
J. Funct. Anal. \textbf{270} (2016), 1888--1921.

\bibitem{c}
A. Almeida and P. H\"ast\"o,
\textit{Besov spaces with variable smoothness and integrability},
J. Funct. Anal. \textbf{258} (2010), 1628--1655.

\bibitem{BMNT10}
\'A. B\'enyi, D. Maldonado, V. Naibo and R. H. Torres,
\textit{On the H\"ormander classes of bilinear pseudodifferential operators},
Integral Equations Operator Theory \textbf{67} (2010), 341--364.

\bibitem{b}
\'A. B\'enyi, F. Bernicot, D. Maldonado, V. Naibo and R. H. Torres,
\textit{On the H\"ormander classes of bilinear pseudodifferential operators II},
Indiana Univ. Math. J. \textbf{62} (2013), 1733--1764.

\bibitem{Besov03}
O. V. Besov,
\textit{Equivalent normings of spaces of functions of variable smoothness},
Proc. Steklov Inst. Math. \textbf{243} (2003), 80--88.

\bibitem{Besov05}
O. V. Besov,
\textit{Interpolation, embedding, and extension of spaces of functions
of variable smoothness},
Proc. Steklov Inst. Math. \textbf{248} (2005), 47--58.

\bibitem{CM75}
R. R. Coifman and Y. Meyer,
\textit{On commutators of singular integrals and bilinear singular integrals},
Trans. Amer. Math. Soc. \textbf{212} (1975), 315--331.

\bibitem{CM78}
R. R. Coifman and Y. Meyer,
\textit{Au-del\`a des op\'erateurs pseudo-diff\'erentiels},
Ast\'erisque \textbf{57}, Soci\'et\'e Math\'ematique de France, Paris, 1978.

\bibitem{DHHMS}
L. Diening, P. Harjulehto, P. H\"ast\"o, Y. Mizuta and T. Shimomura,
\textit{Maximal functions in variable exponent spaces: limiting cases
of the exponent},
Ann. Acad. Sci. Fenn. Math. \textbf{34} (2009), 503--522.

\bibitem{DHR}
L. Diening, P. H\"ast\"o and S. Roudenko,
\textit{Function spaces of variable smoothness and integrability},
J. Funct. Anal. \textbf{256} (2009), 1731--1768.

\bibitem{DB}
D. Drihem and W. Hebbache,
\textit{Boundedness of non-regular pseudodifferential operators on variable
Besov spaces},
J. Pseudo-Differ. Oper. Appl. \textbf{8} (2017), 167--189.

\bibitem{DHC}
D. Drihem and W. Hebbache,
\textit{Continuity of non-regular pseudodifferential operators on variable
Triebel--Lizorkin spaces},
Ann. Polon. Math. \textbf{122} (2019), 233--248.

\bibitem{HN23}
E. Hale and V. Naibo,
\textit{Fractional Leibniz rules in the setting of quasi-Banach function spaces},
J. Fourier Anal. Appl. \textbf{29} (2023), Paper No. 64, 46 pp.

\bibitem{HN}
J. Herbert and V. Naibo,
\textit{Bilinear pseudodifferential operators with symbols in Besov spaces},
J. Pseudo-Differ. Oper. Appl. \textbf{5} (2014), 231--254.

\bibitem{Kempka09}
H. Kempka,
\textit{2-microlocal Besov and Triebel--Lizorkin spaces of variable
integrability},
Rev. Mat. Complut. \textbf{22} (2009), 227--251.

\bibitem{Kempka10}
H. Kempka,
\textit{Atomic, molecular and wavelet decomposition of 2-microlocal
Besov and Triebel--Lizorkin spaces with variable integrability},
Funct. Approx. Comment. Math. \textbf{43} (2010), 171--208.

\bibitem{n}
H. Kempka and J. Vyb\'iral,
\textit{Spaces of variable smoothness and integrability: characterizations
by local means and ball means of differences},
J. Fourier Anal. Appl. \textbf{18} (2012), 852--891.

\bibitem{Leopold89}
H.-G. Leopold,
\textit{On Besov spaces of variable order of differentiation},
Z. Anal. Anwend. \textbf{8} (1989), 69--82.

\bibitem{LT}
Y. Liu and L. Wang,
\textit{The boundedness of bilinear pseudodifferential operators in
Triebel--Lizorkin and Besov spaces with variable exponents},
Bull. Korean Math. Soc. \textbf{61} (2024), 529--540.

\bibitem{11}
J. Marschall,
\textit{Weighted parabolic Triebel spaces of product type, Fourier
multipliers and pseudo-differential operators},
Forum Math. \textbf{3} (1991), 479--511.

\bibitem{12}
J. Marschall,
\textit{Nonregular pseudo-differential operators},
Z. Anal. Anwend. \textbf{15} (1996), 109--148.

\bibitem{NO}
V. Naibo,
\textit{On the bilinear H\"ormander classes in the scales of
Triebel--Lizorkin and Besov spaces},
J. Fourier Anal. Appl. \textbf{21} (2015), 1077--1104.

\bibitem{NT19}
V. Naibo and A. Thomson,
\textit{Coifman--Meyer multipliers: Leibniz-type rules and applications
to scattering of solutions to PDEs},
Trans. Amer. Math. Soc. \textbf{372} (2019), 5453--5481.

\bibitem{T83}
H. Triebel,
\textit{Theory of Function Spaces},
Monographs in Mathematics, Vol. 78,
Birkh\"auser, Basel, 1983.

\bibitem{Vybiral09}
J. Vyb\'iral,
\textit{Sobolev and Jawerth embeddings for spaces with variable
smoothness and integrability},
Ann. Acad. Sci. Fenn. Math. \textbf{34} (2009), 529--544.

\bibitem{Xu08}
J. Xu,
\textit{Variable Besov and Triebel--Lizorkin spaces},
Ann. Acad. Sci. Fenn. Math. \textbf{33} (2008), 511--522.

\bibitem{Xu09}
J. Xu,
\textit{An atomic decomposition of variable Besov and Triebel--Lizorkin
spaces},
Armen. J. Math. \textbf{2} (2009), 1--12.

\bibitem{XE}
J. Xu and J. Zhu,
\textit{Estimates of bilinear pseudodifferential operators associated to
bilinear H\"ormander classes in Besov and Triebel--Lizorkin spaces with
variable exponents},
J. Inequal. Appl. (2018), Paper No. 169, 21 pp.

\end{thebibliography}
\end{document}